\documentclass[11pt]{amsart}
\usepackage[toc,page]{appendix}
\usepackage{latexsym,amsfonts,amssymb,amsthm,amsmath,mathrsfs,color,amscd,graphicx,fullpage,fancyhdr,ulem} 
\usepackage[colorlinks=true,urlcolor=blue,
citecolor=red,linkcolor=blue,linktocpage,pdfpagelabels,bookmarksnumbered,bookmarksopen]{hyperref}
\usepackage{comment}
\usepackage{commath,mathtools,setspace}
\usepackage{paralist}
\usepackage{extarrows}

\excludecomment{uncomment} 

\newtheorem{theorem}{Theorem} 

\newtheorem{lem}[theorem]{Lemma}
\newtheorem{prop}[theorem]{Proposition}

\newtheorem{rem}[theorem]{Remark}

\numberwithin{equation}{section}
\numberwithin{theorem}{section}

\title{Stochastic Control on Space of Random Variables} 

\author[A. Bensoussan]{Alain Bensoussan}
\address{International Center for Decision and Risk Analysis,
	Jindal School of Management, University of Texas at Dallas, Richardson, TX 75083, USA
and School of Data Sciences, City University Hong Kong, Kowloon Tong, Kowloon, Hong Kong}
\email{Alain.Bensoussan@utdallas.edu} 

\author[P.J. Graber]{P. Jameson Graber}
\address{Department of Mathematics, Baylor University, One Bear Place, Waco, TX 97328, USA}
\email{jameson\_graber@baylor.edu} 

\author[S.~C.~P. Yam]{S.~C.~P. Yam}
\address{Department of Statistics, The Chinese University of Hong Kong, Shatin, Hong Kong}
\email{scpyam@sta.cuhk.edu.hk}



\thanks{The first author was supported by  the National Science Foundation under grant
	DMS-1612880, and the Research Grants Council of the Hong Kong Special Administrative Region CityU 113 03 316.\\
	The second author was supported by the National Science Foundation, grant DMS-1612880.\\
	The third author acknowledges the financial support from HKGRF-14300717 with the project title: New Kinds of Forward-Backward Stochastic Systems with Applications, HKSAR-GRF-14301015 with title: Advance in Mean Field Theory, Direct Grant for Research 2014/15 with project code: 4053141 offered by CUHK. He also thanks Columbia University for the kind invitation to be a visiting faculty member in the Department of Statistics during his sabbatical leave. The third author also recalls the unforgettable moments and the happiness shared with his beloved father during the drafting of the present article at their home. Although he lost his father with the deepest sadness at the final stage of the review of this work, his father will never leave the heart of Phillip Yam; and he used this work in memory of his father’s brave battle against liver cancer.} 

\date{\today}

\begin{document}
\maketitle


\begin{abstract}
By extending \cite{bensoussan2015control}, we implement the proposal of Lions \cite{lions14} on studying mean field games and their master equations via certain control problems on the Hilbert space of square integrable random variables. In \cite{bensoussan2015control}, the Hilbert space could be quite general in the face of the ``deterministic control problem'' due to the absence of additional randomness; while the special case of $L^2$ space of square integrable random variables was brought in at the interpretation stage. The effectiveness of the approach was demonstrated by deriving Bellman equations and the first order master equations through control theory of dynamical systems valued in the Hilbert space. In our present problem for second order master equations, it connects with a stochastic control problem over the space of random variables, and it possesses an additional randomness generated by the Wiener process which cannot be detached from the randomness caused by the elements in the Hilbert space. Nevertheless, we demonstrate how to tackle this difficulty, while preserving most of the efficiency of the approach suggested by Lions \cite{lions14}. 
\end{abstract}


\section{INTRODUCTION}
In this paper we study an stochastic optimal control problem on an infinite dimensional Hilbert space, consisting of $L^2$ random variables, and prove that the value function is the unique solution of the corresponding Hamilton-Jacobi-Bellman equation.
The primary motivation is mean field game theory and mean field type control. We refer to \cite{bensoussan2013mean} and references therein for an overview and comparison of the two topics. Mean field games were introduced simultaneously by Caines, Huang, and Malham\'e \cite{huang2006large} as well as Lasry and Lions \cite{lasry07} for the purpose of describing a Nash equilibrium in large population games.
However, as explained in Section 2.6 of \cite{lasry07}, it is common for mean field games to have a \emph{potential,} meaning the equilibrium is at the same time a minimizer for an optimization problem.
Later Lions introduced, in his lectures at the Coll\`ege de France \cite{lions07}, a PDE on the infinite dimensional space of probability measures dubbed ``the master equation" as a way of encoding all the information about the mean field game.
In particular, the solution of the master equation can be seen as the limit of the average value function in an $N$ player Nash equilibrium as $N \to \infty$; this was rigorously proved in \cite{cardaliaguet2015master}.
At the same time, in the case of mean field type control (and thus also for mean field potential games), where one has a corresponding Bellman equation on an infinite dimensional space, the master equation can be derived (at least formally) by differentiating the Bellman equation in the appropriate sense.
This constitutes the central motivation of the present work.

A common approach to studying Bellman equations and the master equation for mean field games and mean field type control has been to use the Wasserstein metric space of probability measures, since for mean field games and mean field type control problems, the key aspect is that the payoffs involve the evolving probability distributions of states.
We refer especially to the work of W.~Gangbo and A.~\'{S}wi\c{e}ch \cite{gangbo2015metric,gangbo2015existence}.
However, a dynamical system whose state space is not a vector space leads to challenging difficulties.
In \cite{lions14} Lions proposed a different approach, in which probability measures are ``lifted" to $L^2$ random variables, which form a Hilbert space.
This also allows a definition of a derivatives in the Wasserstein space, relying on the structure of the gradient in the space of $L^2$ random variables; also see the previously mentioned work of Cardaliaguet, Delarue, Lasry, and Lions \cite{cardaliaguet2015master}.
For Bellman equations, we refer to Pham and Wei \cite{pham2015bellman,pham2017dynamic} and also to Fabbri, Gozzi, and \'{S}wi\c{e}ch \cite{fabbri2017stochastic}.
Inspired by this ``lifting" method, we propose to analyze a control problem entirely on the space of random variables, while the objective functional depends solely on the law of the controlled process, then we can apply our results to mean field type control.
There remains of course the task of interpreting the abstract problem so as to eventually
solve for the mean field type control problem or the mean field game; in particular, one has to check that the dependence of the value function on the random variable is purely through its probability measure, which of course is automatic when one uses the Wasserstein space. As a trade-off, this is much easier than to develop control theory in the Wasssertein space.

In the former paper \cite{bensoussan2015control}, two of the co-authors considered an abstract
control problem for a system whose state space is a Hilbert space which is a purely deterministic one. The fact that the state space is infinite dimensional does not keep the methodology of
control theory from being applicable. In that work, we used a simple set of dynamics, since one of the objectives was to compare with the approach of Gangbo and \'{S}wi\c{e}ch \cite{gangbo2015existence}, when the Hilbert space is the space of random variables. Our approach turns out to be very effective in obtaining the Bellman equation and the master equation of mean field games; see also \cite{bensoussan2015master} and \cite{bensoussan2017interpretation}. 
In the ``deterministic'' case developed in the paper \cite{bensoussan2015control}, one obtains a first order Bellman equation. To address the second order Bellman equation, like those mentioned in the lectures of Lions \cite{lions07} (also see Carmona and Delarue \cite{carmona2017probabilistic,carmona2017probabilisticII}, Bensoussan, Frehse and Yam \cite{bensoussan2017interpretation}), one needs a stochastic control approach. One fundamental difficulty is that we cannot consider a stochastic control problem for a system whose state space is an arbitrary Hilbert space; there is an interaction between the randomness generated by the Wiener process driving the dynamics and the randomness of the elements of the Hilbert space. The main objective of this work is to develop this ``second order'' theory, and to show that it is possible to keep most of the advantages of the deterministic theory, even though the Hilbert space cannot be quite as arbitrary as that in \cite{bensoussan2015control}.

Using the lifting concept introduced by Lions \cite{lions14}, we shall work on the Hilbert space of square integrable random variables and the corresponding notion of G\^ateaux derivatives.
Thanks to the work of Carmona and Delarue \cite{carmona2017probabilistic, carmona2017probabilisticII}, there exist rules relating differentiation over the Hilbert space to so-called ``functional derivatives" over the Wasserstein space of probability measures. 
For first order derivatives, these notions are essentially equivalent.
This is not the case for second order derivatives; nevertheless, when both second order G\^ateaux derivatives and second order functional derivatives exist, we still have transformation formulae to convert one to another, which we call the \emph{rules of correspondence}. Based on these, the advantage of the Hilbert space approach emerges so that a reduced treatment with a direct method can be applied.

We first discuss the formalism in the next section before considering the control problem itself. 

\section{\label{sec:FORMALISM}RULES OF CORRESPONDENCE}
\subsection{WASSERSTEIN SPACE}
We consider the space $\mathcal{P}_{2}(\mathbb{R}^{n})$ of all probability measures on $\mathbb{R}^{n}$ with finite second order moments, equipped with the Wasserstein metric $W_{2}(\mu,\nu)$ defined by: 
\begin{equation*}
W_{2}^{2}(\mu,\nu):=\inf_{\gamma\in\Gamma(\mu,\nu)}\int_{\mathbb{R}^{n}\times \mathbb{R}^{n}}|\xi-\eta|^{2}\gamma(d\xi,d\eta) , 
\end{equation*}
where $\Gamma(\mu,\nu)$ denotes the set of all joint probability measures
on $\mathbb{R}^{n}\times \mathbb{R}^{n}$ such that the marginals are $\mu$ and $\nu$ respectively. Consider an atomless probability space $\left(\Omega, \mathcal{A},\mathbb{P}\right)$ and all its $L^2$ random variables namely, $\mathcal{H}:=L^{2}(\Omega,\mathcal{A},\mathbb{P};\mathbb{R}^{n})$. We then write, for any $X,Y \in \mathcal{H}$, $\mu=\mathcal{L}_{X}$, $\nu=\mathcal{L}_{Y}$, and so 
\begin{equation}
	W_2^2(\mu,\nu)=\inf_{X,Y\in\mathcal{H},\mathcal{L}_{X}=\mu,\mathcal{L}_{Y}=\nu}\mathbb{E}\left(|X-Y|^{2}\right) , \label{eq:2-100}
\end{equation}
where the infimum is attainable, i.e. there is a $(\hat{X}_{\mu}, \hat{X}_{\nu})$ each marginally from $\mathcal{H}$, 
\begin{equation}
W_{2}^2(\mu, \nu)=\mathbb{E}|\hat{X}_{\mu}-\hat{X}_{\nu}|^{2} .\label{eq:101}
\end{equation}
Observe that the map $X \mapsto \mathcal{L}_X$ from $\mathcal{H}$ to $\mathcal{P}_{2}(\mathbb{R}^{n})$ is continuous and surjective, and if we define an equivalence relation in $\mathcal{H}$ by setting $X \sim X' \, \text{if} \, \mathcal{L}_{X}=\mathcal{L}_{X'}$, then the Wasserstein metric is a metric on the quotient space.

\begin{comment}
\textcolor{red}{We also recall the following compactness result for $\mathcal{P}_{2}(\mathbb{R}^{n})$; see Carmona and Delarue \cite{carmona2017probabilistic, carmona2017probabilisticII}:
\begin{equation}
\text{\ensuremath{\{ \mu_k \}_k \subset \mathcal{P}_{2}(\mathbb{R}^{n})} so that \ensuremath{\sup_{k}\mathbb{E}|X_{\mu_{k}}|^{\beta}<\infty}, for a \ensuremath{\beta>2}, is relatively compact.} \label{eq:2-102}
\end{equation}
\textbf{Is this used anywhere??}}
\end{comment}

\subsection{LIFTING PROCEDURE AND FUNCTIONALS}
The lifting proceduce first introduced by P.L. Lions \cite{lions14} consists of regarding a functional $u(m)$ on $\mathcal{P}_{2}(\mathbb{R}^{n})$ as one on $\mathcal{H}$ such that $X\rightarrow u(\mathcal{L}_{X})$; by an abuse of notation, we also denote this functional by $u(X)$, such that $u(X)=u(X')$ whenever $X \sim X'$. The advantage of this approach is that $\mathcal{H}$, unlike $\mathcal{P}_{2}(\mathbb{R}^{n})$, has a Hilbert space structure. Observe that $u(m)$ is continuous with respect to the Wasserstein metric if and only if its lifted functional $u(X)$ is also continuous in $\mathcal{H}$. Indeed, the ``only if" direction follows immediately from definition \eqref{eq:2-100}, while the converse follows from the existence of minimizers in \eqref{eq:101}.

\subsubsection{FIRST ORDER DERIVATIVES}
We next turn to the concept of differentiability.
In $\mathcal{H}$, we can use the standard notion of G\^ateaux derivatives.
	In $\mathcal{P}_{2}(\mathbb{R}^{n})$, we use the notion of functional derivatives: for $u:\mathcal{P}_{2}(\mathbb{R}^{n}) \to \mathbb{R}$ we say that $u$ is differentiable at $m$ if there exists a function, denoted by $\dfrac{\partial u}{\partial m}(m)(x)$, which is continuous in both variables, fulfills 
	\begin{equation}
	\left| \dfrac{\partial u}{\partial m}(m)(x) \right| \leq c(m)(1+|x|^{2}) \label{eq:2-104}
	\end{equation}
	where $c(m)$ is continuous and bounded on bounded subsets of
	$\mathcal{P}_{2}(\mathbb{R}^{n})$, and such that $t\mapsto u(m+t(m'-m))$ is differentiable and for all $m' \in \mathcal{P}_{2}(\mathbb{R}^{n}$ and
\begin{equation}
\dfrac{d}{dt}u(m+t(m'-m))=\int_{\mathbb{R}^{n}}\dfrac{\partial u}{\partial m}(m+t(m'-m))(x)(dm'(x)-dm(x)). \label{eq:2-103}
\end{equation}
We must bear in mind that the functional derivative is unique only up to addition of a function depending on $m$ but constant in $x$.

We now address the relationship between these two notions of derivative.
Let us first assume that $u:\mathcal{P}_{2}(\mathbb{R}^{n}) \to \mathbb{R}$ has functional derivatives.
It does not immediately follow that the lifted version $X \mapsto u(X)$ is G\^ateaux differentiable.
Nevertheless, there is an interesting sufficiency result in Carmona and Delarue \cite{carmona2017probabilistic, carmona2017probabilisticII}: if (i) for each $m$, $x \mapsto \dfrac{\partial u}{\partial m}(m)(x)$ is differentiable; (ii) the gradient $D_{x}\dfrac{\partial u}{\partial m}(m)(x)$ is jointly continuous in $(m,x)$, which is at most of linear growth\footnote{Therefore, $\int_{\mathbb{R}^{n}}\left| D_{x}\dfrac{\partial u}{\partial m}(m)(x)\right|^{2}m(dx)<\infty$.} in $x$ with Lipschitz constant being uniformly bounded in $m$ on bounded sets of $\mathcal{P}_{2}(\mathbb{R}^{n})$, then $u$ is G\^ateaux differentiable so that 
\begin{equation}
D_{X}u(X)=D_{x}\dfrac{\partial u}{\partial m}(\mathcal{L}_{X})(X). \label{eq:2-105}
\end{equation}
Besides $D_{X}u(X) \in \mathcal{H}$, $D_{X}u(X)$ is also $\sigma(X)$-measurable, i.e.~it is a Lebesgue measurable function (from $\mathbb{R}^n$ to itself) of the random variable $X$; moreover, this function depends on $X$ only through its law of $\mathcal{L}_{X}$. These two properties can be made more precise by incorporating the notion of $L$-derivative, denoted by $\partial_{m}u(m)(x)$, as defined in Carmona and Delarue \cite{carmona2017probabilistic, carmona2017probabilisticII} by the formula
\begin{equation}
\partial_{m}u(m)(x)=D_{x}\dfrac{\partial u}{\partial m}(m)(x), \text{ for any $(m,x) \in \mathcal{P}_2(\mathbb{R}^n) \times \mathbb{R}^n$.} \label{eq:2-108}
\end{equation}
Furthermore, we then have
\begin{equation}
D_{X}u(X)=\partial_{m}u(\mathcal{L}_{X})(X) . \label{eq:2-107}
\end{equation}
Note that $x \mapsto D_{x}\dfrac{\partial u}{\partial m}(m)(x)$, unlike the functional derivative itself, is uniquely defined, which
is consistent with the fact that $D_{X}u(X)$ is uniquely defined
as an element in $\mathcal{H}$. 

Conversely, consider a functional $X \mapsto u(X)$ on $\mathcal{H}$, which is G\^ateaux differentiable and depends on $X$ solely through $\mathcal{L}_{X}$; further, if it is uniformly Lipschitz, i.e. for a $C>0$,
\begin{equation}
||D_{X}u(X)-D_{x}u(X')||\leq C||X-X'||, \label{eq:2-109}
\end{equation}
then the ``unlifted" functional $m \mapsto u(m)$ has an $L$-derivative, $\partial_{m}u(m)(x)$, which is (globally) jointly measurable in $(m,x)$ such that 
\begin{equation}
|\partial_{m}u(m)(x)-\partial_{m}u(m)(x')|\leq c|x-x'| .\label{eq:2-110}
\end{equation}
where $c>0$ is a constant independent of $m$; also see \cite{carmona2017probabilistic, carmona2017probabilisticII}. Besides, if $m_{k}\rightarrow m$ in Wasserstein sense, we also have 
\begin{equation}
\partial_{m}u(m_{k})(x)\rightarrow\partial_{m}u(m)(x),\:m-\text{a.e}\:x.\label{eq:2-111}
\end{equation}
In addition, if $(m,x) \mapsto \partial_{m}u(m)(x)$ is jointly continuous, then $u(m)$ has a functional derivative and (\ref{eq:2-108}) is also satisfied. 

We shall refer to \eqref{eq:2-107} (or \eqref{eq:2-105}) as ``the rule of correspondence" between derivatives over the Hilbert space of random variables and over the space of probability distributions.
Through this rule we obtain a synthesis of two different formalisms.
However, as our previous discussion suggests, it is not without limitations: the validity of \eqref{eq:2-107} is only guaranteed under certain assumptions.


\subsubsection{SECOND ORDER DERIVATIVES}
Before proceeding to discuss on the second order derivatives in the two frameworks and their connection, we first provide some useful formulae as a sequel of (\ref{eq:2-103}).
\begin{lem}
Under the assumption that 
\begin{equation}
(m,x) \mapsto D_{x}^{2}\dfrac{\partial u}{\partial m}(m)(x)\:\text{is jointly continuous and bounded, }\label{eq:2-1110}
\end{equation}
we can write $\dfrac{d}{dt}u(m+t(m'-m))$ as
\small{\begin{multline} \label{eq:2-1111} 
\mathbb{E}\left(D_{x}\dfrac{\partial u}{\partial m}(m+t(m'-m))(X_{m}) \cdot (X_{m'}-X_{m}) \right) \\
+  \mathbb{E} \left( \int_{0}^{1}\int_{0}^{1}\alpha D_{x}^{2}\dfrac{\partial u}{\partial m}(m+t(m'-m))(X_{m}+\alpha\beta(X_{m'}-X_{m}))(X_{m'}-X_{m}) \cdot (X_{m'}-X_{m}) d\alpha d\beta \right).
\end{multline}}
\end{lem}
\begin{proof}
Firstly, we can write 
\begin{align*}
\dfrac{d}{dt}u(m+t(m'-m))&=\mathbb{E}\left( \dfrac{\partial u}{\partial m}(m+t(m'-m))(X_{m'})\right)-\mathbb{E} \left(\dfrac{\partial u}{\partial m}(m+t(m'-m))(X_{m}) \right) \\
&= \mathbb{E} \left(\int_{0}^{1}D_{x}\dfrac{\partial u}{\partial m}(m+t(m'-m))(X_{m}+\alpha(X_{m'}-X_{m})) \cdot (X_{m'}-X_{m})d\alpha \right), 
\end{align*}
hence the result follows by applying Taylor's expansion by using the assumption (\ref{eq:2-1110}).
\end{proof}
The situation becomes more complicated if one puts a step further up to the second order level. We first define the second order functional derivatives: the second order functional derivative of a functional $u(m)$ at $m$ is a functional $(m,\xi,\eta) \mapsto \dfrac{\partial^{2}u}{\partial m^{2}}(m)(\xi,\eta)$ such that (i) 
it is jointly continuous and satisfies the following growth condition
\begin{equation} \label{eq:2-113}
\left|\dfrac{\partial^{2}u}{\partial m^{2}}(m)(\xi,\eta)\right|\leq c(m)(1+|\xi|^{2}+|\eta|^{2}) ,
\end{equation}
where $c(m)$ is continuous and is bounded on bounded subsets of $\mathcal{P}_{2}(\mathbb{R}^{n})$; and (ii) $u(m+t(m'-m))$ is twice differentiable in $t$ so that 
\begin{equation} \label{eq:2-112}
\dfrac{d^{2}}{dt^{2}}u(m+t(m'-m))=\int_{\mathbb{R}^{n} \times \mathbb{R}^{n}}\dfrac{\partial^{2}u}{\partial m^{2}}(m+t(m'-m))(\xi,\eta)(dm'(\xi)-dm(\xi))(dm'(\eta)-dm(\eta)).
\end{equation}
From (\ref{eq:2-112}), it is clear that a symmetric version in $(\xi,\eta)$ of $\dfrac{\partial^{2}u}{\partial m^{2}}(m)(\xi,\eta)$ exists, and it is defined up to a function
of the form $c_{1}(m,\xi)+c_{2}(m,\eta)$.
\begin{lem}\label{Lemma2_2} Under the assumption (\ref{eq:2-1110}) and the following:  
\begin{equation} \label{eq:2-115}
(m,\xi,\eta) \mapsto D_{\xi}D_{\eta}\dfrac{\partial^{2}u}{\partial m^{2}}(m)(\xi,\eta)\:\text{is jointly continuous and bounded}, 
\end{equation}
we have the following expression for $u(m')-u(m)$:
\small{\begin{multline} \label{eq:2-117}
\mathbb{E}\left( D_{x}\dfrac{\partial u}{\partial m}(m)(X_{m}) \cdot \Delta X \right) 
 + \mathbb{E}\left( \int_{0}^{1}\int_{0}^{1}\alpha D_{x}^{2}\dfrac{\partial u}{\partial m}(m)(X_{m}+\alpha\beta\Delta X)\Delta X \cdot \Delta Xd\alpha d\beta \right)  \\
 +  \mathbb{E} \left( \int_{0}^{1}\int_{0}^{1}\int_{0}^{1}\int_{0}^{1}tD_{\xi}D_{\eta}\dfrac{\partial^{2}u}{\partial m^{2}}(m+st(m'-m))(X_{m}+\alpha\Delta X,\tilde{X}_{m}+\beta\Delta \tilde X)\Delta \tilde X \cdot \Delta Xdsdtd\alpha d\beta \right)  
\end{multline}}
where $\Delta X := X_{m'}-X_{m}$, and $\tilde{X}_m,\tilde{X}_{m'}$ are independent copies of $X_m,X_{m'}$ respectively.
\end{lem}
\begin{proof}
The proof is put in the appendix A.
\end{proof}
As in the case of first-order derivatives, $D_\xi D_\eta \dfrac{\partial^{2}u}{\partial m^{2}}(m)(\xi,\eta)$ is uniquely determined even though $\dfrac{\partial^{2}u}{\partial m^{2}}(m)(\xi,\eta)$ is not.
Thus all the derivatives appearing in \eqref{eq:2-117} are uniquely defined.

We have already seen that $u(X)=u(\mathcal{L}_X)$ is G\^ateaux differentiable so that $D_{X}u(X)$ is given by \eqref{eq:2-105}. We now discuss the precise notion of the corresponding second order G\^ateaux differential. Consider two elements $X$ and $Y$ from $\mathcal{H}$, and take $m=\mathcal{L}_{X}$ and $m'=\mathcal{L}_{X+\epsilon Y}$, then we can take $X_{m}=X$, $X_{m'}=X+\epsilon Y$,$\:\tilde{X}_{m}=\tilde{X},\:\tilde{X}_{m'}=\tilde{X}+\epsilon\tilde{Y}$, where $(\tilde{X}, \tilde{Y})$ is an independent copy of $(X,Y)$, so that under assumptions (\ref{eq:2-1110}) and (\ref{eq:2-115}), (\ref{eq:2-117}) can be rewritten as:
\begin{align} \label{eq:2-118} 
&u(X+\epsilon Y)
=u(X)+\epsilon\mathbb{E}\left(D_{x}\dfrac{\partial u}{\partial m}(\mathcal{L}_{X})(X) \cdot Y \right) +\epsilon^{2}\mathbb{E}\left(\int_{0}^{1}\int_{0}^{1}\alpha D_{x}^{2}\dfrac{\partial u}{\partial m}(\mathcal{L}_{X})(X+\alpha\beta\epsilon Y)Y\cdot Y d\alpha d\beta \right) \nonumber \\
& \quad +\epsilon^{2}\mathbb{E}\left( \int_{0}^{1}\int_{0}^{1}\int_{0}^{1}\int_{0}^{1}tD_{\xi}D_{\eta}\dfrac{\partial^{2}u}{\partial m^{2}}(\mathcal{L}_{X}+st(\mathcal{L}_{X+\epsilon Y}-\mathcal{L}_{X}))(X+\alpha\epsilon Y,\tilde{X}+\beta\epsilon\tilde{Y})\tilde{Y} \cdot Ydsdtd\alpha d\beta\right). \nonumber
\end{align}
%
Therefore, we obtain that, as $\epsilon\rightarrow0$,
\begin{equation}
\dfrac{u(X+\epsilon Y)-u(X)-\epsilon \mathbb{E}\left(D_{X}u(X)\cdot Y\right)}{\epsilon^{2}}\rightarrow\dfrac{1}{2}\left[\mathbb{E}\left(D_{x}^{2}\dfrac{\partial u}{\partial m}(\mathcal{L}_{X})(X)Y \cdot Y \right) +\mathbb{E}\left(D_{\xi}D_{\eta}\dfrac{\partial^{2}u}{\partial m^{2}}(\mathcal{L}_{X})(X,\tilde{X})\tilde{Y}\cdot Y\right) \right] .\label{eq:2-119}
\end{equation}
The right hand side of (\ref{eq:2-119}) suggests us to define a bilinear continuous functional on $\mathcal{H}$ for each choice of $X \in \mathcal{H}$. For any two elements $Y,Z$ from $\mathcal{H}$, define an independent copy $(\tilde{X},\tilde{Y},\tilde{Z})$ of the triple $(X,Y,Z)$ and a bilinear functional such that 
\begin{equation}
B(X)(Z,Y)=\mathbb{E} \left( D_{x}^{2}\dfrac{\partial u}{\partial m}(\mathcal{L}_{X})(X)Z \cdot Y\right)+\mathbb{E}\left(D_{\xi}D_{\eta}\dfrac{\partial^{2}u}{\partial m^{2}}(\mathcal{L}_{X})(X,\tilde{X})\tilde{Z} \cdot Y \right) \label{eq:2-120}
\end{equation}
\begin{lem}\label{lemma_2_3} For each $X \in \mathcal{H}$, the bilinear form $B(X)(*,*)$ is symmetric, i.e. $B(X)(Z,Y)=B(X)(Y,Z)$.
\end{lem}
\begin{proof}
The proof is put in the appendix A.
\end{proof}
For each $X \in \mathcal{H}$, define the following operator $Z \mapsto \Gamma(X)Z$ in $\mathcal{L}(\mathcal{H},\mathcal{H})$: 
\begin{equation}
\Gamma(X)Z=D_{x}^{2}\dfrac{\partial u}{\partial m}(\mathcal{L}_{X})(X)Z+\mathbb{E}_{\tilde{X},\tilde{Z}}\left(D_{\xi}D_{\eta}\dfrac{\partial^{2}u}{\partial m^{2}}(\mathcal{L}_{X})(X,\tilde{X})\tilde{Z}\right), \label{eq:2-121}
\end{equation}
where the notation $\mathbb{E}_{\tilde{X},\tilde{Z}}$ means taking the expectation with respect to the pair $(\tilde{X},\tilde{Z})$ while freezing the values of $(X,Y,Z)$ (due to the independence property). We can then write $B(X)(Z,Y)=\mathbb{E}\left(\Gamma(X)Z \cdot Y\right)$. Note that the operator norm $\Gamma(X)$ is bounded in $X$ by the assumptions. The convergence (\ref{eq:2-119}) can now be interpreted as: 
\begin{equation}
\dfrac{u(X+\epsilon Y)-u(X)-\epsilon \mathbb{E}\left(D_{X}u(X)\cdot Y \right)}{\epsilon^{2}}\rightarrow\dfrac{1}{2}\mathbb{E}\left(\Gamma(X)Y \cdot Y \right). \label{eq:2-122}
\end{equation}
This convergence serves as the definition of the second order G\^ateaux derivative $D_{X}^{2}u(X) \in \mathcal{L}(\mathcal{H},\mathcal{H})$ so that 
\begin{equation}
D_{X}^{2}u(X)Z= \Gamma(X)Z = D_{x}^{2}\dfrac{\partial u}{\partial m}(\mathcal{L}_{X})(X)Z+\mathbb{E}_{\tilde{X},\tilde{Z}}\left(D_{\xi}D_{\eta}\dfrac{\partial^{2}u}{\partial m^{2}}(\mathcal{L}_{X})(X,\tilde{X})\tilde{Z}\right). \label{eq:2-123}
\end{equation}
Formulae (\ref{eq:2-105}) and (\ref{eq:2-123}) are the \textbf{rules
of correspondence} between the respective concepts of first and second order
G\^ateaux derivatives in the Hilbert space $\mathcal{H}$ and those of first and second order functional derivatives but in $\mathcal{P}_{2}(\mathbb{R}^{n})$.
We have demonstrated that if first or second order functional derivatives in $\mathcal{P}_{2}(\mathbb{R}^{n})$ exist, it is only a matter of regularity so as to also obtain the corresponding first and second order G\^ateaux derivatives in $\mathcal{H}$. However, the reverse is not completely available, although we almost have one for the first order functional derivative. In the rest of this article, we shall develop a full theory of stochastic control in the Hilbert space of $\mathcal{H}$.
We shall then use the rules of correspondence to obtain the corresponding theory for the functional derivatives. It would remain to prove their existence, and we shall give sketchy indications for that goal; it amounts to doing similar calculations as those in $\mathcal{H}$, but much more onerous. 

\section{\label{sec:MEAN-FIELD-TYPE}MEAN FIELD TYPE CONTROL PROBLEMS} Consider a probability space $\left(\Omega,\mathcal{A},\mathbb{P}\right)$ and the Hilbert space $\mathcal{H:=}L^{2}(\Omega,\mathcal{A},\mathbb{P};\mathbb{R}^{n})$ whose inner product is denoted by $((\cdot, \cdot))$, i.e.~$((X,Y)) = \mathbb{E}[X\cdot Y]$, and the corresponding norm is denoted by $\| \cdot \|$. We represent the scalar product in $\mathbb{R}^{n}$ by a dot, as usual. Elements in $\mathcal{H}$ are represented by capital letters, such as $X,Y$, etc., following the tradition in probability theory. We identify $\mathcal{H}$ with its dual.

\subsection{MOTIVATION}
Consider functions $f(x,m)$ and $h(x,m)$ defined on $\mathbb{R}^{n}\times\mathcal{P}_{2}(\mathbb{R}^{n})$ which are associated with the following (law-dependent only) functionals on $\mathcal{P}_{2}(\mathbb{R}^{n})$: 
\begin{align}
\begin{cases} \label{eq:3-100}
F(m)=\int_{\mathbb{R}^{n}}f(x,m)m(dx) ; \\
F_{T}(m)=\int_{\mathbb{R}^{n}}h(x,m)m(dx) .
\end{cases}
\end{align}
Also consider an atomless probability space $(\Omega,\mathcal{A},\mathbb{P})$ with a natural filtration $\mathcal{F}^{t}$ for $t \in [0,T]$, on which a standard $n$-dimensional $\mathcal{F}^{t}$-adapted Wiener process $w(t) \in \mathbb{R}^{n}$ is defined. Define the truncated $\sigma$-field (information set) on $[t,s]$ to be $\mathcal{W}_{t}^{s}:=\sigma(w(\tau)-w(t),\:t \leq\tau\leq s)$, and so the filtration starting from $t$ to $T$, denoted by $\mathcal{W}_{t}$, is $\{\mathcal{W}_{t}^{s}\}_{s \in [t,T]}$. Fix an $m \in \mathcal{P}_{2}(\mathbb{R}^{n})$. We denote a measurable random field element by $v_{x,m,t}(s)$, for $s \in [t,T]$, such that (i) for $m-a.e.$ $x\in \mathbb{R}^{n}$, it is a $\mathcal{W}_{t}$-adapted stochastic process valued in $\mathbb{R}^{n}$; (ii) $\mathbb{E} \{ \int_{t}^{T}\int_{\mathbb{R}^{n}}|v_{x,m,t}(s)|^{2}m(dx)ds\}<+\infty$. We then consider the Hilbert space of all such feasible controls $v_{\cdot,m,t}(\cdot)$ on $[t,T]$ denoted by $L_{\mathcal{W}_{t}}^{2}(t,T;L^{2}(\Omega,\mathcal{A},\mathbb{P};L_{m}^{2}(\mathbb{R}^{n};\mathbb{R}^{n}))$. 
To a control $v(s):=v_{x,m,t}(s)$, we associate a state process starting at $x$ given by
\begin{equation}
x_{x,m,t}(s;v)=x+\int_{t}^{s}v_{x,m,t}(\tau)d\tau+\sigma(w(s)-w(t)), \label{eq:3-110}
\end{equation}
where $\sigma$ is a $n \times n$ (not necessarily invertible) matrix. By construction, $$x_{x,m,t}(s;v) \in L_{\mathcal{W}_{t}}^{2}(t,T;L^{2}(\Omega,\mathcal{A},P;L_{m}^{2}(\mathbb{R}^{n};\mathbb{R}^{n})).$$
To any $m$ so that we can choose a random variable $X_{m,t} \in L^{2}(\Omega,\mathcal{F}^{t},\mathbb{P};\mathbb{R}^{n})$ such that $\mathcal{L}_{X_{m,t}}=m$. 
We define the objective functional on $L_{\mathcal{W}_{t}}^{2}(t,T;L^{2}(\Omega,\mathcal{A},\mathbb{P};L_{m}^{2}(\mathbb{R}^{n};\mathbb{R}^{n}))$ as
\begin{equation}
J_{m,t}(v)=\dfrac{\lambda}{2}\int_{t}^{T}\mathbb{E}|v(s)|^{2}ds+\int_{t}^{T}F(\mathcal{L}_{x_{X_{m,t},m,t}(s;v)})ds+F_{T}(\mathcal{L}_{x_{X_{m,t},m,t}(T;v)}) , \label{eq:3-111}
\end{equation}
where $v = v_{X_{m,t},m,t}$.

Note that $X_{m,t}$ is independent of $\mathcal{W}_{t}^{s}$ for all $s>t$; moreover, the law of $v_{X_{m,t},m,t}(\cdot)$ (like that of $x_{X_{m,t},m,t}(\cdot \, ;v_{X_{m,t},m,t})$) is independent of the particular choice of $X_{m,t}$.
	We denote by $v_{mt}$ the equivalence class of all such processes, and define $L^2_{\mathcal{W}_{m,t}}(t,T; \mathcal{H})$ to be the set of all such equivalence classes, where $\mathcal{W}_{m,t}$ denotes the collection of all filtrations $s \mapsto \sigma(X_{m,t})\vee \mathcal{W}_{t}^{s}$.
	Since each $v_{X_{m,t},m,t}(\cdot)$ is adapted to $s \mapsto \sigma(X_{m,t})\vee \mathcal{W}_{t}^{s}$, we say that the corresponding equivalence class $v_{mt}$ is adapted to $\mathcal{W}_{m,t}$.

Using this formalism, we see that the payoff functional is well-defined for $v \in L^2_{\mathcal{W}_{m,t}}(t,T; \mathcal{H})$, simply by plugging any representative of $v$ into formula \eqref{eq:3-111}.
Thus we define the value function by
\begin{equation}
V(m,t)=\inf_{v\in L^2_{\mathcal{W}_{m,t}}(t,T; \mathcal{H})}J_{m,t}(v). \label{eq:3-112}
\end{equation}

\subsection{HILBERT SPACE OF RANDOM VARIABLES}
We now aim to adapt the lifting procedure to the mean field type control problem (\ref{eq:3-112}). 
Instead of an $m \in \mathcal{P}_2(\mathbb{R}^n)$ and its associated random variable $X_{m,t}$, we take an $X \in L^{2}(\Omega,\mathcal{F}^{t},\mathbb{P};\mathbb{R}^{n})$.
Again, define the truncated filtration on $[t,T]$, denoted by $\mathcal{W}_{X,t}$, as $\{ \sigma(X) \vee \mathcal{W}_{t}^{s} \}_{s \in [t,T]}$. Consider the Hilbert space of processes in $\mathcal{H}$, $L^{2}(t,T;\mathcal{H})$, and its sub-Hilbert space $L_{\mathcal{W}_{X,t}}^{2}(t,T;\mathcal{H})$ which contains all the processes adapted to the filtration $\mathcal{W}_{X,t}$. 
For any control $v \in L_{\mathcal{W}_{X,t}}^{2}(t,T;\mathcal{H})$, define a controlled state $X(\cdot)$ by
\begin{equation} \label{eq:2-51}
X(s) =X+\int_{t}^{s}v(\tau)d\tau+\sigma(w(s)-w(t)) .
\end{equation} 
To be complete, we should write $X_{X,t}(s;v)$ for $X(s)$ to emphasize the dependence on both the initial condition and the control; however, to avoid cumbersome notations, we omit the subscripts if there is no ambiguity.
 
The cost functional is:
\begin{equation} \label{eq:2-52}
J_{X,t}(v): =\frac{\lambda}{2}\int_{t}^{T}||v(s)||^{2}ds+\int_{t}^{T}F(X(s))ds+F_{T}(X(T)) , 
\end{equation}
where $F(X(s))=F(\mathcal{L}_{X(s)})$ and $F_{T}(X(T))=F_{T}(\mathcal{L}_{X(T)})$. 
The value function is given by
\begin{equation}
V(X,t)=\inf_{v \in L_{\mathcal{W}_{X, t}}^{2}(t,T;\mathcal{H})}J_{X,t}(v). \label{eq:2-53}
\end{equation}
We claim that Problem \eqref{eq:2-53} is a lifted version of Problem \eqref{eq:3-112}.
Indeed, suppose $m = \mathcal{L}_X$ and $v \in L^2_{\mathcal{W}_{\mathcal{L}_X,t}}(t,T; \mathcal{H})$.
We can identify $v$ with the particular representative $v_{X,m,t}$.
Then since $F,F_T$ depend on $X$ only through its law, we deduce that $J_{X,t}(v)=J_{m,t}(v)$.
Since $v$ is arbitrary, we have $V(m,t)=V(X,t)$, as desired.

In light of this equivalence between the two optimization problems, our strategy is to show that $V$ satisfies a Bellman equation over $\mathcal{H}$, which we may project down to a PDE over $\mathcal{P}_2$.
However, for such a projection to be valid, we require some assumptions on $F(X)$ and $F_{T}(X)$ and the rules of correspondence listed in Section \ref{sec:FORMALISM}.
Thus the advantage of the lifting approach is that we can work completely within a Hilbert space framework, which simplifies greatly the mathematical development, but the price to pay is that translating the results from one framework to another is nontrivial.

\section{\label{ABSTRACT_CONTROL_PROBLEM}A STUDY ON STOCHASTIC CONTROL PROBLEM (\ref{eq:2-53})}
\subsection{\label{PRELIMIN}PRELIMINARY ASSUMPTIONS}
We then consider functionals $F(X)$ and $F_{T}(X)$ which are continuously G\^ateaux differentiable on $\mathcal{H}$; we also assume that both the gradients $D_{X}F(X)$ and $D_{X}F_{T}(X)$ are Lipschitz continuous:
\begin{align}
\begin{cases} \label{eq:2-1}
& ||D_{X}F(X_{1})-D_{X}F(X_{2})|| \leq c||X_{1}-X_{2}|| ; \\
& ||D_{X}F_{T}(X_{1})-D_{X}F_{T}(X_{2})|| \leq c_{T}||X_{1}-X_{2}|| ,
\end{cases}
\end{align}
where the norms specified on the right hand side of the inequalities are justified by the reflexiveness of $\mathcal{H}$. Besides, we also assume the linear growth of their derivatives and the quadratic growth of the underlying functionals:
\begin{align}
\begin{cases} \label{eq:2-2}
& ||D_{X}F(X)|| \leq C(1+||X||),\quad ||D_{X}F_{T}(X)||\leq C(1+||X||) ; \\
& |F(X)| \leq C(1+||X||^{2}),\quad |F_{T}(X)|\leq C(1+||X||^{2}) . 
\end{cases}
\end{align}
In the above, we denote by $c,c_T$ and $C$ as some generic constants. Moreover, we also assume the quasi-convexity of the functionals:
\begin{align}
\begin{cases} \label{eq:2-3}
& ((D_{X}F(X_{1})-D_{X}F(X_{2}),X_{1}-X_{2})) \geq -c'||X_{1}-X_{2}||^{2} ; \\
& ((D_{X}F_{T}(X_{1})-D_{X}F_{T}(X_{2}),X_{1}-X_{2})) \geq -c'_{T}||X_{1}-X_{2}||^{2}. 
\end{cases}
\end{align}
A simple application of the Cauchy-Schwartz inequality implies that is an effective assumption only when $c'<c$ and $c'_{T}<c_{T}$; otherwise it is automatically fulfilled. We finally make the measurability assumption:  
\begin{equation} \label{eq:2-4}
\text{For each $Y \in \mathcal{H}$, both } D_{X}F(Y) \text{ and } D_{X}F_{T}(Y)\:\text{are \ensuremath{\sigma(Y)}-measurable}. 
\end{equation}
This assumption is satisfied when $F$ and $F_{T}$ depend ``continuously'' and solely on the probability measure of the random variable argument. In other words, although (\ref{eq:2-4}) implies that $D_{X}F(Y)$ and $D_{X}F_{T}(Y)$ are deterministic functions of $Y$, these functions may depend functionally on $X,$ for instance on the probability distribution of $X$, i.e. $D_{X}F(Y)=A_Y(Y)$ for some $A_Y : \mathbb{R}^n \to \mathbb{R}^n$.

\subsection{\label{subsec:CONTROL-PROBLEM}AN OPTIMALITY PRINCIPLE INEQUALITY}
Consider $V(X(t+h),t+h),$ where $X(t+h)$ is given by (\ref{eq:2-51}) with $s=t+h$. We have the flow property: $X_{X(t+h),t+h}(s;v)=X_{X,t}(s;v)$, for $s\geq t+h$. Therefore, for any control $v$,
\begin{align*}
J_{X,t}(v) & =\frac{\lambda}{2}\int_{t}^{t+h}||v(s)||^{2}ds+\int_{t}^{t+h}F(X(s))ds+J_{X(t+h),t+h}(v)\\
 & \geq\frac{\lambda}{2}\int_{t}^{t+h}||v(s)||^{2}ds+\int_{t}^{t+h}F(X(s))ds+V(X(t+h),t+h) ,
\end{align*}
and thus we obtain part of the optimality principle:\begin{equation}
V(X,t)\geq\inf_{v\in L_{\mathcal{W}_{X,t}}^{2}(t,T;\mathcal{H})}\left[\frac{\lambda}{2}\int_{t}^{t+h}||v(s)||^{2}ds+\int_{t}^{t+h}F(X(s))ds+V(X(t+h),t+h)\right] . \label{eq:2-8-1}
\end{equation}

\subsection{G\^ATEAUX DERIVATIVE OF OBJECTIVE FUNCTIONAL}
We shall begin by computing the G\^ateaux derivative of the cost functional $J_{X,t}(v)$.  
\begin{prop}
\label{prop:2-1} Under assumption (\ref{eq:2-1}), the functional $J_{X,t}(v)$ has a G\^ateaux derivative, denoted by $D_{v}J_{X,t}(v)(\cdot) \in L_{\mathcal{W}_{X,t}}^{2}(0,T;\mathcal{H})$, in the space of $L_{\mathcal{W}_{X,t}}^{2}(0,T;\mathcal{H})$ given by the formula: 
\begin{equation}
D_{v}J_{X,t}(v)(s)=\lambda v(s) + \mathbb{E}\left[D_{X}F_{T}(X(T)) + \int_{s}^{T}D_{X}F(X(\tau))d\tau \Bigg| \mathcal{W}_{X,t}^{s}\right] , \, s > t, \label{eq:2-8}
\end{equation}
where $X(s)$ is the state process given by (\ref{eq:2-51}). In addition, it is taken that $D_{v}J_{X,t}(v)(s)=0$ for $s<t$.
\end{prop}
\begin{proof}
The proof is included in the appendix B.
\end{proof}

\subsection{CONVEXITY OF OBJECTIVE FUNCTIONAL}
We next establish the following convexity result: 
\begin{prop}
\label{prop:2-2} Under assumptions (\ref{eq:2-1}), (\ref{eq:2-3}) and 
\begin{equation}
\lambda_{T} := \lambda-c'T-c'_{T}\dfrac{T^{2}}{2}>0 , \label{eq:2-9}
\end{equation}
we then have 
\begin{equation}
\int_{t}^{T}((D_{v}J_{X,t}(v_{1})(s)-D_{v}J_{X,t}(v_{2})(s),v_{1}(s)-v_{2}(s)\,))ds\geq\lambda_{T}\int_{t}^{T}||v_{1}(s)-v_{2}(s)||^{2}ds . \label{eq:2-10}
\end{equation}
\end{prop}
\begin{proof}
The proof is inlcuded in the appendix B.
\end{proof}
Proposition \ref{prop:2-2} implies that the map $v(\cdot)\mapsto D_{v}J_{X,t}(v)(\cdot)$ defines a strictly monotone operator on the Hilbert space $L_{\mathcal{W}_{X,t}}^{2}(t,T;\mathcal{H})$, and consequently the objective functional $J_{X,t}(v)$ is also strictly convex. Moreover, as a consequence of (\ref{eq:2-8}), we obtain
\begin{align*}
\dfrac{d}{d\mu}J_{X,t}(\mu v)  =\int_{t}^{T}((D_{v}J_{X,t}(\mu v)(s),v(s)\,))ds  \geq  \int_{t}^{T}((D_{v}J_{X,t}(0)(s),v(s)\,))ds+\dfrac{\lambda_{T}}{2}\int_{t}^{T}||v(s)||^{2}ds ,
\end{align*}
where the last inequality follows by using (\ref{eq:2-10}) with $v_1=v$ and $v_2=0$. Therefore, integrating against $\mu$ from $0$ to $1$ yields: 
\[
J_{X,t}(v)-J_{X,t}(0) \geq \int_{t}^{T}((D_{v}J_{X,t}(0)(s),v(s)\,))ds + \dfrac{\lambda_{T}}{2}\int_{t}^{T}||v(s)||^{2}ds ,
\]
which implies that $J_{X,t}(v)$ is coercive, i.e. approaching to $\infty$ as $||v||_{L_{\mathcal{W}_{X,t}}^{2}(t,T;\mathcal{H})} \rightarrow \infty$. We can now conclude with:
\begin{prop}
\label{prop:3-3} Under assumptions (\ref{eq:2-1}), (\ref{eq:2-3}) and (\ref{eq:2-9}), the objective functional $J_{X,t}(v)$ has an exactly one minimum point. 
\end{prop}
\begin{rem}
If the functionals $F$ and $F_{T}$ are convex, then $c'=c'_{T}=0$
and the assumption (\ref{eq:2-9}) is automatically fulfilled. Furthermore, for any given value of $\lambda,$ we can interpret (\ref{eq:2-9}) as a smallness condition on $T$.
\end{rem}

\section{\label{sec:STUDY-OF-THE}A STUDY OF THE VALUE FUNCTION}
\subsection{EXPRESSION OF THE VALUE FUNCTION}
Denote by $u(s)$ the optimal control for the objective functional $J_{X,t}(v)$ and by $Y(s)$ the corresponding optimal state. By Proposition \ref{prop:2-1} we have the relation, with initial condition $X$, 
\begin{equation} \label{eq:3-1}
\begin{cases}
& u(s)=-\dfrac{1}{\lambda}Z(s) , \\
& Y(s)=X-\dfrac{1}{\lambda}\int_{t}^{s}Z(\tau)d\tau+\sigma(w(s)-w(t)) , \\
& Z(s)=\mathbb{E}\left[D_{X}F_{T}(Y(T))+\int_{s}^{T}D_{X}F(Y(\tau))d\tau\: \Bigg| \mathcal{W}_{X,t}^{s}\right] .
\end{cases}
\end{equation}
We can also assert that for a given pair $(X,t)$, the system (\ref{eq:3-1}) with unknown adapted processes $(Y(s),Z(s))$ has one and only one solution; while the optimal control is $u(s)=-\dfrac{1}{\lambda}Z(s)$. Sometimes, we may adopt to denote $(Y(s),Z(s))$ as $(Y_{X,t}(s),$ $Z_{X,t}(s))$ so as to emphasize
that these processes are functions of the pair $(X,t)$. The value function is then given by the formula:
\begin{equation}
V(X,t)=\dfrac{1}{2\lambda}\int_{t}^{T}||Z_{X,t}(s)||^{2}ds+\int_{t}^{T}F(Y_{X,t}(s))ds+F_{T}(Y_{X,t}(T)) . \label{eq:3-2}
\end{equation}
Up to the moment, we also remark that assumption (\ref{eq:2-4}) has not really been used; however, it plays a vital role so that there is no gain in enlarging $\mathcal{W}_{X,t}^s$ to $\mathcal{F}^s=\mathcal{F}^t \vee W_t^s$ in (\ref{eq:3-1}):
\begin{prop} \label{prop:3-31} Under assumption (\ref{eq:2-1}),(\ref{eq:2-3}),(\ref{eq:2-4})
and \eqref{eq:2-9}, the following equality holds for all $s \in [t,T]$:
\[
\mathbb{E}\left[D_{X}F_{T}(Y(T))+\int_{s}^{T}D_{X}F(T(\tau))d\tau\: \Bigg| \mathcal{W}_{X,t}^{s}\right]=\mathbb{E}\left[ D_{X}F_{T}(Y(T))+\int_{s}^{T}D_{X}F(Y(\tau))d\tau\: \Bigg| \mathcal{F}^{s}\right] .
\]
\end{prop}
\begin{proof}
Following the arguments in the previous paragraphs, there also exists a unique pair $(\tilde{Y}(s),\tilde{Z}(s))$ such that 
\begin{equation} \label{eq:3-21}
\begin{cases}
& \tilde{Y}(s)=X-\dfrac{1}{\lambda}\int_{t}^{s}\tilde{Z}(\tau)d\tau+\sigma(w(s)-w(t)) , \\
& \tilde{Z}(s)=\mathbb{E}\left[D_{X}F_{T}(\tilde{Y}(T))+\int_{s}^{T}D_{X}F(\tilde{Y}(\tau))d\tau\: \Bigg| \mathcal{F}^{s}\right] .
\end{cases}
\end{equation}
The pair $(\tilde{Y}(s),\tilde{Z}(s))$ is adapted to $\mathcal{F}^{s}$. Thanks to assumption (\ref{eq:2-4}), we can assert that $D_{X}F_{T}(\tilde{Y}(T))+\int_{s}^{T}D_{X}F(\tilde{Y}(\tau))d\tau$ is $\mathcal{W}_{X,t}^{T}$-measurable. On the other hand, $\mathcal{W}_{X,t}^{T}=\mathcal{W}_{X,t}^{s}\vee \mathcal{W}_s^T$ and $\mathcal{W}_s^T$ is independent of $\mathcal{W}_{X,t}^{s}$. Since $\mathcal{W}_{X,t}^{T}\subseteq\mathcal{F}^{T}$, $\mathcal{F}^{T}=\mathcal{F}^{s}\vee \mathcal{W}_s^T$ while $\mathcal{W}_s^T$ is also independent of $\mathcal{F}^{s}$ (it is an {\em innovation}) by definition, meanwhile $\mathcal{W}_{X,t}^{s}\subseteq\mathcal{F}^{s}$, we must necessarily \footnote{Recall the elementary result that for two independent $\sigma$-fields $\mathcal{G}$ and $\mathcal{H}$, if $V$ is a random variable independent of $\mathcal{H}$, then $\mathbb{E}(V | \sigma(\mathcal{G}, \mathcal{H}))= \mathbb{E}(V | \mathcal{G})$.} conclude with the claimed equality in the statement. 
\end{proof}
As a remark, the pair $(\tilde{Y}(s),\tilde{Z}(s))$ in \eqref{eq:3-21} is also a solution of \eqref{eq:3-1} and since the solution of \eqref{eq:3-1} is unique, we must have $Y(s)=\tilde{Y}(s)$ and $Z(s)=\tilde{Z}(s)$.

\subsection{GROWTH OF OPTIMAL SOLUTION AND VALUE FUNCTION}
We here obtain the bounds for the optimal solution and the value function:
\begin{prop} \label{prop:3-4} Under assumptions \eqref{eq:2-1},\eqref{eq:2-3},\eqref{eq:2-4} and \eqref{eq:2-9}, we have the bounds: 
\begin{equation}\label{eq:3-3}
||Y_{X,t}(s)||, ||Z_{X,t}(s)|| \leq C(1+||X||) \quad \text{ and } \quad |V(X,t)| \leq C(1+||X||^{2}) ,
\end{equation}
where the constant $C$ depends only on $T,\lambda_{T}$
and the constants of the problem \eqref{eq:2-51}, \eqref{eq:2-52} and \eqref{eq:2-53}. 
\end{prop}
\begin{proof}
The proof is included in the appendix C.
\end{proof}

\subsection{GRADIENT AND SMOOTHNESS OF VALUE FUNCTION}
Our objective now is to establish the regularity of the gradient of the value function. 
\begin{theorem} \label{Theo3-1} Under the assumptions specified in Proposition \ref{prop:3-4}, the value function $V(X,t)$ is continuously G\^ateaux differentiable in $X$, $D_{X}V(X,t)=Z_{X,t}(t)$ and it is $\sigma(X)$-measurable. Moreover, $D_{X}V(X,t)$ is Lipschitz continuous in $X$, and particularly we have the estimates:
\begin{equation}
||D_{X}V(X,t)||\leq C(1+||X||) \text{ and } ||D_{X}V(X_{1},t)-D_{X}V(X_{2},t)||\leq C||X_{1}-X_{2}|| , \label{eq:3-31}
\end{equation}
where $C$ is a constant depending on the constants of the model.
\end{theorem}
\begin{proof}
The proof is enclosed in the appendix C.
\end{proof}

\subsection{DYNAMIC PROGRAMMING PRINCIPLE FOR THE VALUE FUNCTION}
We now complete the optimality principle as follows. 
\begin{prop} \label{prop:3-5} Under the assumptions specified in Proposition \ref{prop:3-4}, the optimality principle is given by 
\begin{align} \label{eq:3-8}
\begin{cases}
& V(X,t) =\dfrac{1}{2\lambda}\int_{t}^{t+h}||Z_{X,t}(s)||^{2}ds+\int_{t}^{t+h}F(Y_{X,t}(s))ds+V(Y_{X,t}(t+h),t+h),\,\text{ for } t+h\leq T, \\
& V(X,T) =F_{T}(X) . 
\end{cases}
\end{align}
\end{prop}
\begin{proof} Again we omit unnecessary subscripts $X$ and $t$. We have $V(X,t)=J_{X,t}(u)$ with $u(s)=-\dfrac{1}{\lambda}Z(s)$. According to (\ref{eq:2-8-1}) in Section \ref{subsec:CONTROL-PROBLEM}, we can assert that 
\[
V(X,t)\geq\dfrac{1}{2\lambda}\int_{t}^{t+h}||Z(s)||^{2}ds+\int_{t}^{t+h}F(Y(s))ds+V(Y(t+h),t+h) .
\]
 On the other hand, if we take the control $\tilde{u}(s)$ which is optimal for the problem with initial condition $(Y(t+h),t+h)$, then combining $u(s)$ for $s\in(t,t+h)$ and $\tilde{u}(s)$ for $s\in(t+h,T)$ we get an admissible control for the problem with initial condition $(X,t)$. The corresponding cost is $\dfrac{1}{2\lambda}\int_{t}^{t+h}||Z(s)||^{2}ds+\int_{t}^{t+h}F(Y(s))ds+V(Y(t+h),t+h)$ which is greater in value than $V(X,t)$. Therefore, the reverse inequality holds which implies (\ref{eq:3-8}). 
\end{proof}
It follows that 
\[
V(Y(t+h),t+h)=\dfrac{1}{2\lambda}\int_{t+h}^{T}||Z(s)||^{2}ds+\int_{t+h}^{T}F(Y(s))ds+F_{T}(Y(T)) ,
\] 
and the pair $(Y(s),Z(s))$ is also the solution of the system (\ref{eq:3-1}) corresponding to initial condition $(Y(t+h),\,t+h)$. Therefore, $Z_{X,t}(t+h)=D_{X}V(Y(t+h),t+h)$ and in general $Z_{X,t}(s)=-\dfrac{1}{\lambda}D_{X}V(Y_{X,t}(s),s)$. The control $u(s)$ is optimal for the problem with initial condition $(Y(t+h),t+h)$. 

\subsection{REGULARITY IN TIME OF VALUE FUNCTION}
We aim to show the following regularity result in time for the value function.  
\begin{prop} \label{prop:3-6} Under the assumptions specified in Proposition \ref{prop:3-4}, for any $t \leq t_{1}, t_{2} \leq T$, suppose that $X$ is both $\mathcal{F}^{t_{1}}$- and $\mathcal{F}^{t_{2}}$-measurable, then we have the estimate:
\begin{equation}
|V(X,t_{2})-V(X,t_{1})|\leq C(1+||X||^{2})|t_{2}-t_{1}| .\label{eq:3-9}
\end{equation}
In addition, $D_{X}V(X,t)$ is (H\"older-)continuous in time such that
\begin{equation}
||D_{X}V(X,t_{2})-D_{X}V(X,t_{1})||\leq C(1+||X||)|t_{2}-t_{1}|^{\frac{1}{2}} . \label{eq:3-90}
\end{equation}
\end{prop}
\begin{proof}
The proof is put in the appendix C.
\end{proof}

\section{SECOND ORDER DERIVATIVE OF VALUE FUNCTION}
\subsection{ASSUMPTIONS AND SECOND ORDER DERIVATIVE}
To get more regularity for the value function, we need more assumptions. We now assume that $D_{X}^{2}F(X)\in\mathcal{L}(\mathcal{H};\mathcal{H})$
and $D_{X}^{2}F_{T}(X)\in\mathcal{L}(\mathcal{H};\mathcal{H})$ and
satisfy
\begin{equation}
\begin{array}{c}
||D_{X}^{2}F(X)||\leq c,\,||D_{X}^{2}F_{T}(X)||\leq c_{T},\\
((D_{X}^{2}F(X)\Xi,\Xi))+c'||\Xi||^{2}\geq0,\:((D_{X}^{2}F_{T}(X)\Xi,\Xi))+c'_{T}||\Xi||^{2}\geq0,\:\forall\Xi\in\mathcal{H}.
\end{array}\label{eq:4-1}
\end{equation}
 Also we make the measurability assumption 
\begin{equation}
D_{X}^{2}F(X),\:D_{X}^{2}F_{T}(X)\:\text{are \ensuremath{\sigma(X)} measurable.}\label{eq:4-101}
\end{equation}
This last assumption has to be explained, since these linear operators
are not matrices. What makes sense is $D_{X}^{2}F(X)Z$ for any $Z\in\mathcal{H}$
and the map $Z\rightarrow D_{X}^{2}F(X)Z$ is linear from $\mathcal{H}$
to $\mathcal{H}.$ The assumption (\ref{eq:4-101}) means 
\begin{equation}
D_{X}^{2}F(X)Z=B_{X}(X)Z+C_{XZ}(X)\label{eq:4-102}
\end{equation}
where
$$
B_X:\mathbb{R}^{n} \to \mathcal{L}(\mathbb{R}^{n};\mathbb{R}^{n})
\ \ \text{and} \ \
C_{XZ}: \mathbb{R}^{n} \to\mathbb{R}^{n}
$$
are deterministic functions.
Moreover the map $Z\rightarrow C_{XZ}(x)$ is linear. We also have that
$D_{X}^{2}F(X)Z$ is $\sigma(X,Z)$ measurable. 

The assumptions \eqref{eq:4-1} are naturally compatible with the
Lipschitz asssumptions \eqref{eq:2-1}, \eqref{eq:2-3}. We shall
assume also the H\"older regularity property 
\begin{equation}
\begin{array}{c}
||D_{X}^{2}F(X_{1})-D_{X}^{2}F(X_{2})||\leq c||X_{1}-X_{2}||^{\delta}\\
||D_{X}^{2}F_{T}(X_{1})-D_{X}^{2}F_{T}(X_{2})||\leq c_{T}||X_{1}-X_{2}||^{\delta},0<\delta\leq1.
\end{array}\label{eq:4-2}
\end{equation}
 We want to prove the following regularity of the value function 
\begin{theorem}
\label{theo4-1}We make the assumptions of Proposition $\ref{prop:3-4}$
and \eqref{eq:4-1}, \eqref{eq:4-2}. Then the value function $V(X,t)$
is twice continuously differentiable $in$ $X$ and $D_{X}^{2}V(X,t)\mathcal{X}$
for $X,\mathcal{X}$ $\mathcal{F}^{t}$ measurable is $\mathcal{F}^{t}$
measurable, and in fact $\sigma(X,\mathcal{X})$ measurable. Moreover, we have the H\"older regularity property 
\begin{equation}
||D_{X}^{2}V(X_{1},t)-D_{X}^{2}V(X_{2},t)||\leq C||X_{1}-X_{2}||^{\delta}\label{eq:4-3}
\end{equation}
where $C$ is a generic constant.
\end{theorem}

\subsection{\label{sub:LINEAR-QUADRATIC-CONTROL}LINEAR QUADRATIC CONTROL PROBLEM}

To construct the second derivative, we introduce a linear quadratic
control problem as follows. Let $X,t$ initial conditions as usual
and $X$ is $\mathcal{F}^{t}$ measurable. We consider the optimal
trajectory $Y_{Xt}(s)$ and the corresponding process $Z_{Xt}(s),$
which we denote $Y(s)$ and $Z(s)$ as above. Let $\mathcal{X}$ in
$\mathcal{H}$ which is $\mathcal{F}^{t}$ measurable. 
We define
the following linear quadratic control problem. For a control $\mathcal{V}(s)$
adapted to $\mathcal{F}^{s}$ we consider the state $\mathcal{X}(s)$
defined by 
\begin{equation}
\mathcal{X}(s)=\mathcal{X}+\int_{t}^{s}\mathcal{V}(\tau)d\tau\label{eq:4-4}
\end{equation}
and the payoff 
\begin{multline}
\mathcal{J}_{\mathcal{X}t}(\mathcal{V}(.))=\dfrac{\lambda}{2}\int_{t}^{T}||\mathcal{V}(s)||^{2}ds+\dfrac{1}{2}\int_{t}^{T}((D_{X}^{2}F(Y(s))\mathcal{X}(s),\mathcal{X}(s)))ds\\
+\dfrac{1}{2}((D_{X}^{2}F_{T}(Y(T))\mathcal{X}(T),\mathcal{X}(T))).\label{eq:4-5}
\end{multline}
 Thanks to the assumption on $\lambda$ (see (\ref{eq:2-9})) the cost
functional is quadratic convex and the problem has a unique minimum
denoted by $\mathcal{U}(s)$. If $\mathcal{Y}(s)$ denotes the optimal
state we have the system of necessary and sufficient conditions
\begin{equation}
\begin{array}{c}
\mathcal{Y}(s) =\mathcal{X}-\dfrac{1}{\lambda}\int_{t}^{s}\mathcal{Z}(\tau)d\tau\\
\mathcal{Z}(s)=\mathbb{E}[D_{X}^{2}F_{T}(Y(T))\mathcal{Y}(T)+\int_{s}^{T}D_{X}^{2}F(Y(\tau))\mathcal{Y}(\tau)d\tau|\mathcal{F}^{s}]
\end{array}\label{eq:4-6}
\end{equation}
and the optimal control is $\mathcal{U}(s)=-\dfrac{1}{\lambda}\mathcal{Z}(s).$
We consider in particular 
\begin{equation}
\mathcal{Z}(t)=\mathbb{E}[D_{X}^{2}F_{T}(Y(T))\mathcal{Y}(T)+\int_{s}^{T}D_{X}^{2}F_{T}(Y(\tau))\mathcal{Y}(\tau)d\tau|\mathcal{F}^{t}]\label{eq:4-7}
\end{equation}
 and the map $\mathcal{X}\rightarrow\mathcal{Z}(t)$ defines an operator
$\mathcal{\varUpsilon}(t)\mathcal{X}=\mathcal{Z}(t),$ which belongs
to $\mathcal{L}(\mathcal{H};\mathcal{H}).$ Our objective is to check
that $\mathcal{\varUpsilon}(t)=D_{X}^{2}V(X,t).$ Because of the conditional
expectation, which is a projection in the Hilbert space
$\mathcal{H}$, we cannot write an explicit formula for the operator,
independently of the argument. The pair $\mathcal{Y}(s),\mathcal{Z}(s)$
is in fact adapted to the filtration $\mathcal{W}_{X\mathcal{X}t}^{s}=\sigma(X,\mathcal{X},w(\tau)-w(t),t\leq\tau\leq s).$
We already know by Proposition \ref{prop:3-31} that $Y(s),Z(s)$
are adapted to the filtration $\mathcal{W}_{X,t}^{s}$. We define $\mathcal{\tilde{Y}}(s),$$\mathcal{\tilde{Z}}(s)$
as $\mathcal{Y}(s),\mathcal{Z}(s),$ but swapping $\mathcal{F}^{s}$
with $\mathcal{W}_{X\mathcal{X}t}^{s}.$ Then by the measurability
assumption (\ref{eq:4-101}), (\ref{eq:4-102}) the random variable
$D_{X}^{2}F_{T}(Y(T))\mathcal{Y}(T)+\int_{s}^{T}D_{X}^{2}F_{T}(Y(\tau))\mathcal{Y}(\tau)d\tau$
is $\mathcal{W}_{X\mathcal{X}t}^{T}$ measurable. Reasoning as in
Proposition \ref{prop:3-31}, we conclude that $\mathcal{\tilde{Y}}(s)=\mathcal{Y}(s),\:\mathcal{\tilde{Z}}(s)=\mathcal{Z}(s).$
Hence the adaptability property. 

The next important result is 
\begin{equation}
\inf_{\mathcal{V}(.)}\mathcal{J}_{\mathcal{X}t}(\mathcal{V}(.))=\dfrac{1}{2}((\mathcal{\varUpsilon}(t)\mathcal{X},X)),\label{eq:4-8}
\end{equation}
whose proof is a standard exercise in quadratic optimization. We finally give bounds. 
\begin{prop}
\label{prop:4-1}We assume (\ref{eq:4-1}),(\ref{eq:4-2}) and the
assumptions of Proposition \ref{prop:3-4}. We then have the estimates 
\begin{equation}
\sup_{t\leq s\leq T}||\mathcal{Y}(s)||,\:\sup_{t\leq s\leq T}||\mathcal{Z}(s)||\leq C||\mathcal{X}||\label{eq:4-9}
\end{equation}
 In particular $||\mathcal{\varUpsilon}(t)||\leq C.$ where $C$ is
a generic constant. \end{prop}
\begin{proof}
We use 
\begin{multline}
((\mathcal{Z}(t),\mathcal{X}))=\cfrac{1}{\lambda}\int_{t}^{T}||\mathcal{Z}(s)||^{2}ds+\int_{t}^{T}((D_{X}^{2}F(Y(s))\mathcal{Y}(s),\mathcal{Y}(s)))ds\\+((D_{X}^{2}F_{T}(Y(T))\mathcal{Y}(T),\mathcal{Y}(T)))
\geq\cfrac{1}{\lambda}\int_{t}^{T}||\mathcal{Z}(s)||^{2}ds-c'\int_{t}^{T}||\mathcal{Y}(s)||^{2}ds-c'_{T}||\mathcal{Y}(T)||^{2}\label{eq:4-90}
\end{multline}
and by calculations similar to those previously done 
\begin{equation}
((\mathcal{Z}(t),\mathcal{X}))\geq\dfrac{1}{\lambda}(1-\dfrac{1}{\lambda}(1+\epsilon)T(c'_{T}+\dfrac{c'T}{2}))\int_{t}^{T}||\mathcal{Z}(s)||^{2}ds-(1+\dfrac{1}{\epsilon})(c'_{T}+c'T)||\mathcal{X}||^{2}
\label{eq:4-10}
\end{equation}
But from (\ref{eq:4-8}) we have $\dfrac{1}{2}((\mathcal{Z}(t),\mathcal{X}))\leq\mathcal{J}_{\mathcal{X}t}(0)=\dfrac{1}{2}[\int_{t}^{T}((D_{X}^{2}F(Y(s))\mathcal{X},\mathcal{X}))ds+((D_{X}^{2}F_{T}(Y(T))\mathcal{X},\mathcal{X}))].$
Therefore $((\mathcal{Z}(t),\mathcal{X}))\leq(c_{T}+cT)||\mathcal{X}||^{2}$.
Combining with (\ref{eq:4-10}) we obtain $\int_{t}^{T}||\mathcal{Z}(s)||^{2}ds\leq C||\mathcal{X}||^{2}.$
From this estimate and formulas (\ref{eq:4-6}) we deduce  the
estimates (\ref{eq:4-9}). 
\end{proof}

\subsection{PROOF OF THEOREM \ref{theo4-1}}

Consider two random variables $X_{1},X_{2}$ which are $\mathcal{F}^{t}$
measurable. As in the proof of Theorem \ref{Theo3-1} we
associate the pairs $Y_{1}(s),Z_{1}(s)$ and $Y_{2}(s),Z_{2}(s)$
corresponding to the optimal control problems with initial conditions
$X_{1},t$ and $X_{2},t$. We also consider the associated linear
quadratic control problems defined in Section \ref{sub:LINEAR-QUADRATIC-CONTROL}.
They also depend on the choice of initials conditions. Now consider the linear control problem related to $Y_{2}(s),Z_{2}(s)$
with initial condition $\mathcal{X}=X_{1}-X_{2}.$ We call its optimal
solution $\mathcal{Y}_{12}(s),\,\mathcal{Z}_{12}(s)$. Precisely 
\begin{equation}
\begin{array}{c}
\mathcal{Y}_{12}(s)=X_{1}-X_{2}-\dfrac{1}{\lambda}\int_{t}^{s}\mathcal{Z}_{12}(\tau)d\tau,\\
\mathcal{Z}_{12}(s)=\mathbb{E}[D_{X}^{2}F_{T}(Y_{2}(T))\mathcal{Y}_{12}(T)+\int_{s}^{T}D_{X}^{2}F(Y_{2}(\tau))\mathcal{Y}_{12}(\tau)d\tau|\mathcal{F}_{t}^{s}].
\end{array}
\label{eq:4-11}
\end{equation}
Note that $\mathcal{Z}_{12}(t)=\mathcal{\varUpsilon}_{2}(t)(X_{1}-X_{2}).$
We next define the trajectory $Y_{12}(s)=Y_{2}(s)+\mathcal{Y}_{12}(s).$
It satisfies the equation 
\begin{equation*}
Y_{12}(s)=X_{1}-\dfrac{1}{\lambda}\int_{t}^{s}(Z_{2}(\tau)+\mathcal{Z}_{12}(\tau))d\tau+\sigma(w(s)-w(t)).
\end{equation*}
 We call $\mathcal{U}_{12}(s)=-\dfrac{1}{\lambda}\mathcal{Z}_{12}(s)$. So the trajectory $Y_{12}(s)$ corresponds to the control $u_{2}(s)+\mathcal{U}_{12}(s)$
and the initial condition $X_{1}.$ We thus have $V(X_{1},t)\leq J_{X_{1}t}(u_{2}(.)+\mathcal{U}_{12}(.)).$
Therefore 
\begin{multline*}
V(X_{1},t)-V(X_{2},t)\leq\dfrac{1}{2\lambda}\int_{t}^{T}||Z_{2}(s)+\mathcal{Z}_{12}(s)||^{2}ds-\dfrac{1}{2\lambda}\int_{t}^{T}||Z_{2}(s)||^{2}ds\\
+\int_{t}^{T}(F(Y_{2}(s)+\mathcal{Y}_{12}(s))-F(Y_{2}(s)))ds+F_{T}(Y_{2}(T)+\mathcal{Y}_{12}(T))-F_{T}(Y_{2}(T))
\end{multline*}
 From the assumptions \eqref{eq:4-2} we have the estimates
{\small \begin{equation*}
\begin{array}{c}
|F(Y_{2}(s)+\mathcal{Y}_{12}(s))-F(Y_{2}(s))-((D_{X}F(Y_{2}(s)),\mathcal{Y}_{12}(s)))-\dfrac{1}{2}((D_{X}^{2}F(Y_{2}(s))\mathcal{Y}_{12}(s),\mathcal{Y}_{12}(s)))|\leq C||\mathcal{Y}_{12}(s)||^{2+\delta},\\
|F_{T}(Y_{2}(T)+\mathcal{Y}_{12}(T))-F_{T}(Y_{2}(T))-((D_{X}F_{T}(Y_{2}(T)),\mathcal{Y}_{12}(T)))-\dfrac{1}{2}((D_{X}^{2}F_{T}(Y_{2}(T))\mathcal{Y}_{12}(T),\mathcal{Y}_{12}(T)))|\leq C||\mathcal{Y}_{12}(T)||^{2+\delta}.
\end{array}
\end{equation*}}
From the estimates \eqref{eq:4-9} we have $||\mathcal{Y}_{12}(s)||\leq C||X_{1}-X_{2}||$.
We also use 
\begin{equation*}
\int_{t}^{T}((Z_{2}(s),\mathcal{Z}_{12}(s)))ds=\int_{t}^{T}((D_{X}F_{T}(Y_{2}(T))+\int_{s}^{T}D_{X}F(Y(\tau))d\tau,\mathcal{Z}_{12}(s)))ds.
\end{equation*}
 Combining and rearranging we obtain 
\begin{equation}
V(X_{1},t)-V(X_{2},t)\leq((X_{1}-X_{2},Z_{2}(t)))+\dfrac{1}{2}((\varUpsilon_{2}(t)(X_{1}-X_{2}),X_{1}-X_{2}))+C||X_{1}-X_{2}||^{2+\delta}\label{eq:4-12}
\end{equation}
Interchanging the roles of $X_{1},X_{2}$ leads to 
\begin{equation}
V(X_{1},t)-V(X_{2},t)\geq((X_{1}-X_{2},Z_{1}(t)))-\dfrac{1}{2}((\varUpsilon_{1}(t)(X_{1}-X_{2}),X_{1}-X_{2}))-C||X_{1}-X_{2}||^{2+\delta}\label{eq:4-13}
\end{equation}

To proceed, we need a precise estimate of $Z_{1}(t)-Z_{2}(t).$
We introduce for $\theta\in(0,1)$ the system 
\begin{equation}
\begin{array}{c}
Y^{\theta}(s)=X_{1}+\theta(X_{2}-X_{1})-\dfrac{1}{\lambda}\int_{t}^{s}Z^{\theta}(\tau)d\tau+\sigma(w(s)-w(t)),\\
Z^{\theta}(s)=\mathbb{E}[D_{X}F_{T}(Y^{\theta}(T))+\int_{s}^{T}D_{X}F(Y^{\theta}(\tau))d\tau|\mathcal{F}_{t}^{s}].
\end{array}\label{eq:4-14}
\end{equation}
We note that if we interchange the roles of $X_{1}$ and $X_{2}$
then we obtain $Y^{1-\theta}(s).$ We next define $Y'^{\theta}(s),Z'^{\theta}(s)$
(the notation means that they are the derivatives of $Y^{\theta}(s),Z^{\theta}(s)$
with respect to $\theta$) by the system
\begin{equation}
\begin{array}{c}
Y'^{\theta}(s)=X_{2}-X_{1}-\dfrac{1}{\lambda}\int_{t}^{s}Z'^{\theta}(\tau)d\tau,\\
Z'^{\theta}(s)=\mathbb{E}[D_{X}^{2}F_{T}(Y^{\theta}(T))Y'^{\theta}(T)+\int_{s}^{T}D_{X}^{2}F(Y^{\theta}(\tau))Y'^{\theta}(\tau)d\tau|\mathcal{F}_{t}^{s}].
\end{array}\label{eq:4-15}
\end{equation}
We see that $Y^{0}(s)=Y_{1}(s)$ and $Y^{1}(s)=Y_{2}(s).$ Also recalling
the definition of $\mathcal{Y}_{12}(s),\mathcal{Z}_{12}(s),$ see
(\ref{eq:4-11}). In particular $\mathcal{Z}_{21}(t)=\varUpsilon_{1}(t)(X_{2}-X_{1}).$
Define finally $\mathcal{Y}_{21}^{\theta}(s)=Y'^{\theta}(s)-\mathcal{Y}_{21}(s),$
$\mathcal{Z}_{21}^{\theta}(s)=Z'^{\theta}(s)-\mathcal{Z}_{21}(s)$. We have the relations 
\begin{multline}
\mathcal{Y}_{21}^{\theta}(s)=-\dfrac{1}{\lambda}\int_{t}^{s}\mathcal{Z}_{21}^{\theta}(\tau)d\tau,\label{eq:4-16}\\
\mathcal{Z}_{21}^{\theta}(s)=\mathbb{E}[D_{X}^{2}F_{T}(Y_{1}(T))\mathcal{Y}_{21}^{\theta}(T)+(D_{X}^{2}F_{T}(Y^{\theta}(T))-D_{X}^{2}F_{T}(Y_{1}(T)))Y'^{\theta}(T)+
\\
+\int_{s}^{T}(D_{X}^{2}F(Y_{1}(\tau))\mathcal{Y}_{21}^{\theta}(\tau)+(D_{X}^{2}F(Y^{\theta}(\tau))-D_{X}^{2}F(Y_{1}(\tau)))Y'^{\theta}(\tau))d\tau\:|\mathcal{F}_{t}^{s}].
\end{multline}
From (\ref{eq:4-90}) we have 
\begin{equation}
((\mathcal{Z}_{21}(t),X_{2}-X_{1}))=\dfrac{1}{\lambda}\int_{t}^{T}||\mathcal{Z}_{21}(s)||^{2}ds+\int_{t}^{T}((D_{X}^{2}F(Y_{1}(s))\mathcal{Y}_{21}(s),\mathcal{Y}_{21}(s)))ds+((D_{X}^{2}F_{T}(Y_{1}(T))\mathcal{Y}_{21}(T),\mathcal{Y}_{21}(T)))\label{eq:4-17}
\end{equation}
Next, we check that 
\begin{equation}
Z_{2}(s)-Z_{1}(s)=\mathcal{Z}_{21}(s)+\int_{0}^{1}\mathcal{Z}_{21}^{\theta}(s)d\theta.\label{eq:4-18}
\end{equation}
Similarly, we introduce $Y'^{1-\theta}(s),Z'^{1-\theta}(s)$. We
have $Y'^{1}(s)=\mathcal{Y}_{12}(s),Z'^{1}(s)=\mathcal{Z}_{12}(s)$
and we define $\mathcal{Y}_{12}^{1-\theta}(s)=Y'^{1-\theta}(s)-\mathcal{Y}_{12}(s),\mathcal{Z}_{12}^{1-\theta}(s)=Z'^{1-\theta}(s)-\mathcal{Z}_{12}(s).$
We have relations similar to (\ref{eq:4-16}), (\ref{eq:4-17}). Moreover, as in (\ref{eq:4-18}) we have 
\begin{equation}
Z_{1}(s)-Z_{2}(s)=\mathcal{Z}_{12}(s)+\int_{0}^{1}\mathcal{Z}_{12}^{1-\theta}(s)d\theta\label{eq:4-19}
\end{equation}
Therefore, in particular, 
\begin{multline*}
((Z_{1}(t)-Z_{2}(t),X_{1}-X_{2}))=((Z_{2}(t)-Z_{1}(t),X_{2}-X_{1}))=\\
((\mathcal{\varUpsilon}_{2}(t)(X_{1}-X_{2}),X_{1}-X_{2}))+((\int_{0}^{1}\mathcal{Z}_{12}^{1-\theta}(s)d\theta,X_{1}-X_{2}))=\\
((\mathcal{\varUpsilon}_{1}(t)(X_{2}-X_{1}),X_{2}-X_{1}))+((\int_{0}^{1}\mathcal{Z}_{21}^{\theta}(s)d\theta,X_{2}-X_{1})).
\end{multline*}
 This can be written as 
\begin{multline*}
((Z_{1}(t)-Z_{2}(t),X_{1}-X_{2}))=\dfrac{1}{2}(((\mathcal{\varUpsilon}_{1}(t)+\mathcal{\varUpsilon}_{2}(t))(X_{1}-X_{2}),X_{1}-X_{2}))\\
+\dfrac{1}{2}((\int_{0}^{1}\mathcal{Z}_{12}^{1-\theta}(s)d\theta,X_{1}-X_{2}))+\dfrac{1}{2}((\int_{0}^{1}\mathcal{Z}_{21}^{\theta}(s)d\theta,X_{2}-X_{1})).
\end{multline*}
If we go back to (\ref{eq:4-13}) we can assert that
\begin{multline}
V(X_{1},t)-V(X_{2},t)\geq((X_{1}-X_{2},Z_{2}(t)))+\dfrac{1}{2}((\varUpsilon_{2}(t)(X_{1}-X_{2}),X_{1}-X_{2}))-C||X_{1}-X_{2}||^{2+\delta}
\\
+\dfrac{1}{2}((\int_{0}^{1}\mathcal{Z}_{12}^{1-\theta}(s)d\theta,X_{1}-X_{2}))+\dfrac{1}{2}((\int_{0}^{1}\mathcal{Z}_{21}^{\theta}(s)d\theta,X_{2}-X_{1}))-C||X_{1}-X_{2}||^{2+\delta}.\label{eq:4-20}
\end{multline}
From the definition of $Y'^{\theta}(s),$$Z'^{\theta}(s),$ see (\ref{eq:4-15})
we obtain by already used techniques 
\begin{equation}
\sup_{t\leq s\leq T}||Y'^{\theta}(s)||\leq C||X_{1}-X_{2}||,\:\sup_{t\leq s\leq T}||Z'^{\theta}(s)||\leq C||X_{1}-X_{2}||.\label{eq:4-21}
\end{equation}
Since $Y_{1}(s)=Y^{0}(s)$ we can assert that 
\begin{equation*}
Y^{\theta}(s)-Y_{1}(s)=\int_{0}^{\theta}Y'^{\lambda}(s)d\lambda
\end{equation*}
and thus from \eqref{eq:4-21} we can state 
\begin{equation}
\sup_{t\leq s\leq T}||Y^{\theta}(s)-Y_{1}(s)||\leq C||X_{1}-X_{2}||.\label{eq:4-22}
\end{equation}
From the assumption (\ref{eq:4-2}) it follows that
\begin{equation*}
\begin{array}{c}
||D_{X}^{2}F_{T}(Y^{\theta}(T))-D_{X}^{2}F_{T}(Y_{1}(T))||\leq C||X_{1}-X_{2}||^{\delta},\\
||D_{X}^{2}F(Y^{\theta}(s))-D_{X}^{2}F(Y_{1}(s))||\leq C||X_{1}-X_{2}||^{\delta},
\end{array}
\end{equation*}
and thus from \eqref{eq:4-21} we get 
\begin{equation*}
||(D_{X}^{2}F(Y^{\theta}(s))-D_{X}^{2}F(Y_{1}(s)))Y'^{\theta}(s)||\leq C||X_{1}-X_{2}||^{1+\delta},\\
||(D_{X}^{2}F_{T}(Y^{\theta}(T))-D_{X}^{2}F_{T}(Y_{1}(T)))Y'^{\theta}(T)||\leq C||X_{1}-X_{2}||^{1+\delta}.
\end{equation*}
Looking at \eqref{eq:4-16} we now obtain  
\begin{equation}
\sup_{t\leq s\leq T}||\mathcal{Z}_{21}^{\theta}(s)||\leq C||X_{1}-X_{2}||^{1+\delta}\label{eq:4-23}
\end{equation}
and similarly 
\begin{equation}
\sup_{t\leq s\leq T}||\mathcal{Z}_{12}^{1-\theta}(s)||\leq C||X_{1}-X_{2}||^{1+\delta}\label{eq:4-24}
\end{equation}
Combining (\ref{eq:4-12}) and (\ref{eq:4-20}) it follows that 
\begin{equation}
|V(X_{1},t)-V(X_{2},t)-((X_{1}-X_{2},Z_{2}(t)))-\dfrac{1}{2}((\varUpsilon_{2}(t)(X_{1}-X_{2}),X_{1}-X_{2}))|\leq C||X_{1}-X_{2}||^{2+\delta}\label{eq:4-25}
\end{equation}
which we may rewrite as 
\begin{equation}
|V(X+\mathcal{X},t)-V(X,t)-((\mathcal{X},Z(t)))-\dfrac{1}{2}((\varUpsilon(t)\mathcal{X},\mathcal{X}))|\leq C||\mathcal{X}||^{2+\delta}.\label{eq:4-26}
\end{equation}
This proves that $V$ is twice continuously differentiable in $X,$
with $D_{X}^{2}V(X,t)\mathcal{=\varUpsilon}(t)$ and $D_{X}^{2}V(X,t)\mathcal{X}$
is $\sigma(X,\mathcal{X})$ measurable. We finally prove the H\"older
estimate (4.3). 
Let $X_{1},X_{2}$ be $\mathcal{F}_{t}$ measurable
and $\mathcal{X}$ be $\mathcal{F}_{t}$ measurable. We define $Y_{1}(s),Z_{1}(s)$
and $Y_{2}(s),Z_{2}(s)$ associated with the initial conditions $X_{1},t$
and $X_{2},t$ respectively. We then define $\mathcal{Y}_{1}(s),\mathcal{Z}_{1}(s)$
and $\mathcal{Y}_{2}(s),\mathcal{Z}_{2}(s)$ respectively by (\ref{eq:4-6}).
We have $D_{X}^{2}V(X_{1},t)\mathcal{X}=\varUpsilon_{1}(t)\mathcal{X}=\mathcal{Z}_{1}(t)$
and $D_{X}^{2}V(X_{2},t)\mathcal{X}=\varUpsilon_{2}(t)\mathcal{X}=\mathcal{Z}_{2}(t)$.
Setting $\mathcal{\tilde{Y}}(s)=\mathcal{Y}_{1}(s)-\mathcal{Y}_{2}(s)$
and $\mathcal{\tilde{Z}}(s)=\mathcal{Z}_{1}(s)-\mathcal{Z}_{2}(s)$
we obtain 
\begin{multline*}
\mathcal{\tilde{Z}}(s)=\mathbb{E}[D_{X}^{2}F_{T}(Y_{1}(T))\mathcal{Y}_{1}(T)-D_{X}^{2}F_{T}(Y_{2}(T))\mathcal{Y}_{2}(T)\\
+\int_{s}^{T}(D_{X}^{2}F(Y_{1}(\tau))\mathcal{Y}_{1}(\tau)-D_{X}^{2}F(Y_{2}(\tau))\mathcal{Y}_{2}(\tau))d\tau\,|\mathcal{F}_{t}^{s}].
\end{multline*}
 We note that $||\mathcal{Y}_{1}(s)||,\,$$||\mathcal{Y}_{2}(s)||\leq C||\mathcal{X}||$
and 
\[
||(D_{X}^{2}F_{T}(Y_{1}(T))-D_{X}^{2}F_{T}(Y_{2}(T)))\mathcal{Y}_{1}(T)||\leq C||X_{1}-X_{2}||^{\delta}||\mathcal{X}||,
\]
\[
||(D_{X}^{2}F(Y_{1}(s))-D_{X}^{2}F(Y_{2}(s)))\mathcal{Y}_{1}(s)||\leq C||X_{1}-X_{2}||^{\delta}||\mathcal{X}||,
\]
 from which, using techniques already used, it follows that
\[
\sup_{t\leq s\leq T}||\mathcal{\tilde{Y}}(s)||,\,\sup_{t\leq s\leq T}||\mathcal{\tilde{Z}}(s)||\,\leq C||X_{1}-X_{2}||^{\delta}||\mathcal{X}||
\]
 and the result (\ref{eq:4-3}) is obtained immediately. This completes
the proof. We leave to the reader to check directly that $D_{X}^{2}V(X,t)$
is self-adjoint. 

\section{MAIN RESULT }

\subsection{PRELIMINARIES}

We first begin with a result which bears similarities with the result
of Proposition \ref{prop:3-31}. We state the 
\begin{prop}
\label{prop:5-1}We make the assumptions of Proposition \ref{prop:4-1}
and (\ref{eq:4-101}). Let $X$ be $\mathcal{F}^{t}$ measurable.
Then 
\begin{equation}
|((D_{X}^{2}V(X,t+h)\sigma(\dfrac{w(t+h)-w(t)}{h^{1/2}}),\sigma(\dfrac{w(t+h)-w(t)}{h^{1/2}}))-((D_{X}^{2}V(X,t)\sigma N,\sigma N))|\leq Ch^{\frac{\delta}{2}}(1+||X||^{2\delta})\label{eq:5-1}
\end{equation}
where $N$ is a standard gaussian variable in $\mathbb{R}^{n}$ which is independent
of $\mathcal{F}^{0}$ and the Wiener process. \end{prop}
\begin{proof}
Although $D_{X}^{2}V(X,t)\mathcal{X}$ has been defined only on arguments
$\mathcal{X}$ which are $\mathcal{F}^{t}$ measurable, we can also
take an initial condition like $\sigma N$ ($\sigma$ is there for
convenience), where $N$ is independent of the filtration $\mathcal{F}^{s}$.
We replace the system (\ref{eq:4-6}) by 
\begin{equation}
\begin{array}{c}
\mathcal{Y}(s) =\mathcal{\sigma}N-\dfrac{1}{\lambda}\int_{t}^{s}\mathcal{Z}(\tau)d\tau,\\
\mathcal{Z}(s)=\mathbb{E}[D_{X}^{2}F_{T}(Y(T))\mathcal{Y}(T)+\int_{s}^{T}D_{X}^{2}F(Y(\tau))\mathcal{Y}(\tau)d\tau|\mathcal{F}^{s}\cup\sigma(N)].
\end{array}
\label{eq:5-2}
\end{equation}
Note that $N$ is independent of the process $Y(s)$. Consider now
$D_{X}^{2}V(X,t+h)\sigma(\dfrac{w(t+h)-w(t)}{h^{1/2}})=\mathcal{Z}_{h}(t+h)$
where the pair $\mathcal{Z}_{h}(s),\mathcal{Y}_{h}(s)$ is defined
by 
\begin{equation}
\begin{array}{c}
\mathcal{Y}_{h}(s)=\mathcal{\sigma}(\dfrac{w(t+h)-w(t)}{h^{1/2}})-\dfrac{1}{\lambda}\int_{t+h}^{s}\mathcal{Z}_{h}(\tau)d\tau,\\
\mathcal{Z}_{h}(s)=\mathbb{E}[D_{X}^{2}F_{T}(Y^{h}(T))\mathcal{Y}_{h}(T)+\int_{s}^{T}D_{X}^{2}F(Y^{h}(\tau))\mathcal{Y}_{h}(\tau)d\tau|\mathcal{F}^{s}]
\end{array}\label{eq:5-3}
\end{equation}
for $s\geq t+h.$ Also $Y^{h}(s),Z^{h}(s)$ are defined by the system
\begin{equation}
\begin{array}{c}
Y^{h}(s)=X-\dfrac{1}{\lambda}\int_{t}^{s}Z^{h}(\tau)d\tau+\sigma(w(s)-w(t+h))\\
Z^{h}(s)=\mathbb{E}[D_{X}F_{T}(Y^{h}(T))+\int_{s}^{T}D_{X}F(Y^{h}(\tau))d\tau|\mathcal{F}^{s}]
\end{array}\label{eq:5-4}
\end{equation}
and we know that $Y^{h}(s),Z^{h}(s)$ are adapted to the filtration
$\mathcal{W}_{X,t+h}^{s}$ and thus are independent of $w(\tau)-w(t),\;\forall t\leq\tau\leq t+h.$
Moreover $\mathcal{Y}_{h}(s),\mathcal{Z}_{h}(s)$ are adapted to $\mathcal{W}_{X,w(t+h)-w(t),t+h}^{s}$. Recall that
\begin{multline}
((D_{X}^{2}V(X,t+h)\sigma(\dfrac{w(t+h)-w(t)}{h^{1/2}}),\sigma(\dfrac{w(t+h)-w(t)}{h^{1/2}}))=\dfrac{1}{\lambda}\int_{t+h}^{T}||\mathcal{Z}_{h}(s)||^{2}ds\\
+\int_{t+h}^{T}((D_{X}^{2}F(Y^{h}(s))\mathcal{Y}_{h}(s),\mathcal{Y}_{h}(s)))ds+((D_{X}^{2}F_{T}(Y^{h}(T))\mathcal{Y}_{h}(T),\mathcal{Y}_{h}(T))).
\label{eq:5-5}
\end{multline}
We do not change the value of the right hand side by replacing $\dfrac{w(t+h)-w(t)}{h^{1/2}}$
by a fixed $N$ which is standard Gaussian independent of $\mathcal{F}^{s}.$
We have 
\begin{equation}
((D_{X}^{2}V(X,t+h)\sigma(\dfrac{w(t+h)-w(t)}{h^{1/2}}),\sigma(\dfrac{w(t+h)-w(t)}{h^{1/2}}))=((D_{X}^{2}V(X,t+h)\sigma N,\sigma N))\label{eq:5-51}
\end{equation}
 with 
\begin{multline}\label{eq:5-6}
((D_{X}^{2}V(X,t+h)\sigma N,\sigma N))=\dfrac{1}{\lambda}\int_{t+h}^{T}||\mathcal{Z}_{h}(s)||^{2}ds \\
+\int_{t+h}^{T}((D_{X}^{2}F(Y^{h}(s))\mathcal{Y}_{h}(s),\mathcal{Y}_{h}(s)))ds+((D_{X}^{2}F_{T}(Y^{h}(T))\mathcal{Y}_{h}(T),\mathcal{Y}_{h}(T)))
\end{multline}
 and
\begin{align}
\mathcal{Y}_{h}(s)&=\mathcal{\sigma}N-\dfrac{1}{\lambda}\int_{t+h}^{s}\mathcal{Z}_{h}(\tau)d\tau,\label{eq:5-7}\\
\mathcal{Z}_{h}(s)&=\mathbb{E}[D_{X}^{2}F_{T}(Y^{h}(T))\mathcal{Y}_{h}(T)+\int_{s}^{T}D_{X}^{2}F(Y^{h}(\tau))\mathcal{Y}_{h}(\tau)d\tau|\mathcal{F}^{s}\cup\sigma(N)]. \notag
\end{align}
 We next estimate the difference $|((D_{X}^{2}V(X,t+h)\sigma N,\sigma N))-((D_{X}^{2}V(X,t)\sigma N,\sigma N))|.$
We have $Y^{h}(s)=Y_{X,t+h}(s),\,Z^{h}(s)=Z_{X,t+h}(s)$ and by
\eqref{eq:3-10} we obtain
\begin{equation}
\sup_{t+h\leq s\leq T}||Y^{h}(s)-Y(s)||,\:\sup_{t+h\leq s\leq T}||Z^{h}(s)-Z(s)||\leq C(1+||X||)h^{\frac{1}{2}}\label{eq:5-9}
\end{equation}
Recalling that $\mathcal{Y}(s),\mathcal{Z}(s)$ is the solution of
(\ref{eq:5-2}) we define $\mathcal{\tilde{Y}}_{h}(s)=\mathcal{Y}_{h}(s)-\mathcal{Y}(s),$
$\mathcal{\tilde{Z}}_{h}(s)=\mathcal{Z}_{h}(s)-\mathcal{Z}(s)$. After
calculations already done, we get 
\begin{multline*}
\dfrac{1}{\lambda}\int_{t+h}^{T}||\mathcal{\tilde{Z}}_{h}(s)||^{2}ds+((D_{X}^{2}F_{T}(Y(T))\mathcal{\tilde{Y}}_{h}(T),\mathcal{\tilde{Y}}_{h}(T)))+\int_{t+h}^{T}((D_{X}^{2}F(Y(s))\mathcal{\tilde{Y}}_{h}(s),\mathcal{\tilde{Y}}_{h}(s)))ds\\
+(((D_{X}^{2}F_{T}(Y^{h}(T))-D_{X}^{2}F_{T}(Y(T)))\mathcal{Y}_{h}(T),\mathcal{\tilde{Y}}_{h}(T)))+\int_{t+h}^{T}(((D_{X}^{2}F(Y^{h}(s))-D_{X}^{2}F(Y(s)))\mathcal{Y}_{h}(s),\mathcal{\tilde{Y}}_{h}(s)))ds=\\
((\sigma N-\mathcal{Y}(t+h),\mathcal{\tilde{Z}}_{h}(t+h))).
\end{multline*}
 Using the assumption (\ref{eq:4-2}) and the estimates (\ref{eq:5-9})
we deduce
\begin{equation*}
\begin{array}{c}
||(D_{X}^{2}F_{T}(Y^{h}(T))-D_{X}^{2}F_{T}(Y(T)))\mathcal{Y}_{h}(T)||\leq C(1+||X||)^{\delta}h^{\frac{\delta}{2}},\\
||(D_{X}^{2}F(Y^{h}(s))-D_{X}^{2}F(Y(s)))\mathcal{Y}_{h}(s)||\leq C(1+||X||)^{\delta}h^{\frac{\delta}{2}}.
\end{array}
\end{equation*}
 Using $\mathcal{\tilde{Y}}_{h}(s)=\sigma N-\mathcal{Y}(t+h)-\dfrac{1}{\lambda}\int_{t+h}^{s}\mathcal{\tilde{Z}}_{h}(\tau)d\tau,$
$||\mathcal{\tilde{Z}}_{h}(t+h)||\leq C(||\mathcal{\tilde{Y}}_{h}(T)||+\int_{t+h}^{T}||\mathcal{\tilde{Y}}_{h}(s)||ds+(1+||X||)^{\delta}h^{\frac{\delta}{2}}),$
and performing standard estimates we can obtain 
\begin{equation*}
\sup_{t+h\leq s\leq T}||\mathcal{\tilde{Y}}_{h}(s)||,\:\sup_{t+h\leq s\leq T}||\mathcal{\tilde{Z}}_{h}(s)||\leq C(1+||X||)^{\delta}h^{\frac{\delta}{2}}.
\end{equation*}
 Finally,
\begin{multline*}
((D_{X}^{2}V(X,t+h)\sigma N,\sigma N))-((D_{X}^{2}V(X,t)\sigma N,\sigma N))=-\dfrac{1}{\lambda}\int_{t}^{t+h}||\mathcal{Z}(s)||^{2}ds-\int_{t}^{t+h}((D_{X}^{2}F(Y(s))\mathcal{Y}(s),\mathcal{Y}(s)))ds\\
+\int_{t+h}^{T}[((D_{X}^{2}F(Y^{h}(s))\mathcal{Y}_{h}(s),\mathcal{Y}_{h}(s)))-((D_{X}^{2}F(Y(s))\mathcal{Y}(s),\mathcal{Y}(s)))]ds\\
+((D_{X}^{2}F_{T}(Y^{h}(T))\mathcal{Y}_{h}(T),\mathcal{Y}_{h}(T)))-((D_{X}^{2}F_{T}(Y(T))\mathcal{Y}(T),\mathcal{Y}(T))),
\end{multline*}
and from previous estimates and standard arguments we obtain the
estimate (\ref{eq:5-1}). This concludes the proof. 
\end{proof}

\subsection{BELLMAN EQUATION}

We can now state the main result 
\begin{theorem}
We make the assumptions of Proposition \ref{eq:5-1}. The value function
$V(X,t)$ is differentiable in $t$, with H\"older continuous derivative
and is a solution of the Bellman equation 
\begin{align}
\dfrac{\partial V}{\partial t}+\dfrac{1}{2}((D_{X}^{2}V(X,t)\sigma N,\sigma N))-\dfrac{1}{2\lambda}||D_{X}V(X,t)||^{2}+F(X)=0,\label{eq:5-10}\\
\nonumber
V(X,T)=F_{T}(X)
\end{align}
where $N$ is a standard Gaussian, which is independent of the filtration
$\mathcal{F}^{s},$ and in particular of $X$.
\end{theorem}
\begin{proof}
We go back to the optimality principle (\ref{eq:3-8}). We write 
\begin{equation}
V(X,t)-V(X,t+h)=\dfrac{1}{2\lambda}\int_{t}^{t+h}||Z(s)||^{2}ds+\int_{t}^{t+h}F(Y(s))ds+V(Y(t+h),t+h)-V(X,t+h)\label{eq:5-11}
\end{equation}
Next, using Theorem $\ref{theo4-1},$ we write 
\begin{multline*}
V(Y(t+h),t+h)-V(X,t+h)=((D_{X}V(X,t+h),Y(t+h)-X))+((D_{X}^{2}V(X,t+h)(Y(t+h)-X),Y(t+h)-X))\\
+\int_{0}^{1}\int_{0}^{1}\theta(((D_{X}^{2}V(X+\theta\mu(Y(t+h)-X),t+h)-D_{X}^{2}V(X,t+h))(Y(t+h)-X),Y(t+h)-X))d\theta d\mu.
\end{multline*}
But
\begin{equation*}
((D_{X}V(X,t+h),Y(t+h)-X))=-\dfrac{1}{\lambda}((D_{X}V(X,t+h),\int_{t}^{t+h}Z(s)ds))
\end{equation*}
since $D_{X}V(X,t+h)$ is independent of $w(t+h)-w(t)$. From \eqref{eq:4-3}
we can assert that 
\begin{multline*}
|(((D_{X}^{2}V(X+\theta\mu(Y(t+h)-X),t+h)-D_{X}^{2}V(X,t+h))(Y(t+h)-X),Y(t+h)-X))|
\leq C||Y(t+h)-X||^{1+\delta}.
\end{multline*}
Also, thanks to \eqref{eq:5-51},
\begin{multline*}
((D_{X}^{2}V(X,t+h)(Y(t+h)-X),Y(t+h)-X))=\dfrac{1}{\lambda^{2}}((D_{X}^{2}V(X,t+h)\int_{t}^{t+h}Z(s)ds,\int_{t}^{t+h}Z(s)ds))
\\+h((D_{X}^{2}V(X,t+h)\sigma N,\sigma N)),
\end{multline*}
so 
\begin{multline}
|\dfrac{1}{h}(V(X,t)-V(X,t+h))-\dfrac{1}{2\lambda}\dfrac{\int_{t}^{t+h}||Z(s)||^{2}}{h}ds-\dfrac{\int_{t}^{t+h}F(Y(s))ds}{h}\label{eq:5-12} \\
+\dfrac{1}{\lambda}((D_{X}V(X,t+h),\dfrac{\int_{t}^{t+h}Z(s)ds}{h}))-((D_{X}^{2}V(X,t+h)\sigma N,\sigma N))|\leq C(1+||X||^{2})h+Ch^{\delta}.
\end{multline}
We have $|\dfrac{\int_{t}^{t+h}F(Y(s))ds}{h}-F(X)|\leq C(1+||X||^{2})h.$
Next $Z(s)=Z_{Xt}(s)$ and according to (\ref{eq:3-101}) we have
$||Z(s)-Z_{Xs}(s)||\leq C(1+||X||)(s-t)^{\frac{1}{2}}$. Moreover
$Z_{Xs}(s)=D_{X}V(X,s),$ so we must estimate the difference $D_{X}V(X,s)-D_{X}V(X,t).$
To do that, we go to (\ref{eq:3-8}) and we differentiate in $X.$
This is possible in view of the smoothness which is available.
We can start by considering $X+\theta\tilde{X}$ with $\tilde{X}$
$\mathcal{F}^{t}$ measurable and differentiate in $\theta.$ We obtain
the formula 
\begin{multline}
((D_{X}V(X,t)-D_{X}V(X,t+h),\tilde{X}))=\frac{1}{\lambda}\int_{t}^{t+h}((Z(s),\mathcal{\tilde{Z}}(s)))ds+\label{eq:5-13}\\
+\int_{t}^{t+h}((D_{X}F(Y(s)),\mathcal{\tilde{Y}}(s)))ds-\dfrac{1}{\lambda}((D_{X}V(Y(t+h),t+h),\int_{t}^{t+h}\mathcal{\tilde{Z}}(s)ds))\\
+((D_{X}V(Y(t+h),t+h)-D_{X}V(X,t+h),\tilde{X}))
\end{multline}
with $\mathcal{\tilde{Y}}(s),\mathcal{\tilde{Z}}(s)$ defined by
\begin{equation}
\begin{array}{c}
\mathcal{\tilde{Y}}(s)=\tilde{X}-\dfrac{1}{\lambda}\int_{t}^{s}\mathcal{\tilde{Z}}(\tau)d\tau\\
\mathcal{\tilde{Z}}(s)=\mathbb{E}[D_{X}^{2}F_{T}(Y(T))\mathcal{\tilde{Y}}(T)+\int_{s}^{T}D_{X}^{2}F(Y(\tau))\mathcal{\tilde{Y}}(\tau)d\tau|\mathcal{F}^{s}]
\end{array}
\label{eq:5-14}
\end{equation}
and after some calculations which are now standard, we obtain 
\begin{equation*}
|((D_{X}V(X,t)-D_{X}V(X,t+h),\tilde{X}))|\leq Ch(1+||X||)||\tilde{X}||+Ch^{\frac{1}{2}}||\tilde{X}||.
\end{equation*}
Noting that $D_{X}V(X,t+h)$ is $\mathcal{F}^{t}$ measurable, the
above inequality implies 
\begin{equation}
||D_{X}V(X,t)-D_{X}V(X,t+h)||\leq Ch(1+||X||)+Ch^{\frac{1}{2}}.\label{eq:5-15}
\end{equation}
This implies $||Z(s)-D_{X}V(X,t)||\leq C(1+||X||)h^{\frac{1}{2}}.$
Using established estimates we obtain 
\begin{equation*}
|-\dfrac{1}{2\lambda}\dfrac{\int_{t}^{t+h}||Z(s)||^{2}}{h}ds+\dfrac{1}{\lambda}((D_{X}V(X,t+h),\dfrac{\int_{t}^{t+h}Z(s)ds}{h}))-\dfrac{1}{2\lambda}||D_{X}V(X,t)||^{2}|\leq C(1+||X||^{2})h^{\frac{1}{2}}.
\end{equation*}
Finally, from \eqref{eq:5-51} and \eqref{eq:5-1} we have 
\begin{equation}
|((D_{X}^{2}V(X,t+h)\sigma N,\sigma N))-((D_{X}^{2}V(X,t)\sigma N,\sigma N))|\leq Ch^{\frac{\delta}{2}}(1+||X||^{2\delta}).\label{eq:5-150}
\end{equation}
Collecting results we can assert that 
\begin{multline}
|\dfrac{1}{h}(V(X,t)-V(X,t+h))+\dfrac{1}{2\lambda}||D_{X}V(X,t)||^{2}-F(X)\label{eq:5-16} \\
-((D_{X}^{2}V(X,t)\sigma N,\sigma N))|\leq Ch^{\frac{\delta}{2}}(1+||X||^{2\delta})
\end{multline}
and $V(X,t)$ is differentiable in $t,$ with derivative given by
equation \eqref{eq:5-10}. Now from the estimates \eqref{eq:5-15}
and \eqref{eq:5-150}, it follows imediately from the equation \eqref{eq:5-10}
that $\dfrac{\partial V}{\partial t}$ is H\"older continuous in $t.$
The proof has been completed.
\end{proof}

\section{THE MASTER EQUATION }

\subsection{DEFINITION AND PRELIMINARIES}

The master equation concerns the equation for $\mathcal{U}(X,t)=D_{X}V(X,t).$
So it is a function from $\mathcal{H}\times(0,T)\rightarrow\mathcal{H}$.
We know already that with the definition of the filtration $\mathcal{F}^{t}$
above, for $X$ $\mathcal{F}^{t}$ measurable, then $\mathcal{U}(X,t)$
is $\mathcal{F}^{t}$ measurable. In fact $\mathcal{U}(X,t)$ is $\sigma(X)$
measurable. Moreover, by (\ref{eq:4-3}) this function is differentiable
in $X,$ with H\"older derivative
\begin{equation}
||D_{X}\mathcal{U}(X_{1},t)-D_{X}\mathcal{U}(X_{2},t)||\leq C||X_{1}-X_{2}||^{\delta}\label{eq:6-1}
\end{equation}
 and $D_{X}\mathcal{U}(X,t)$ $\in\mathcal{L}(\mathcal{H};\mathcal{H})$,
self adjoint. We have also shown in (\ref{eq:5-15}) that $\mathcal{U}(X,t)$
is H\"older in $t.$ To obtain further regularity, we need further assumptions.
We shall assume the existence of $D_{X}^{3}F(X),$ $D_{X}^{3}F_{T}(X)$.
These are objects in $\mathcal{L}(\mathcal{H};\mathcal{L}(\mathcal{H};\mathcal{H})).$
So if $\Xi,\Upsilon$ are in $\mathcal{H},$ the value $D_{X}^{3}F(X)\Xi$
belongs to $\mathcal{L}(\mathcal{H};\mathcal{H}).$ We can then consider
the result of this linear map on $\Upsilon,$ denoted $D_{X}^{3}F(X)\Xi\Upsilon$
which is an element of $\mathcal{H}$. We shall assume 
\begin{equation}
\begin{array}{c}
||D_{X}^{3}F(X_{1})\Xi\Upsilon-D_{X}^{3}F(X_{2})\Xi\Upsilon||\leq c||X_{1}-X_{2}||^{\delta}||\Xi||\,||\Upsilon||,\\
||D_{X}^{3}F_{T}(X_{1})\Xi\Upsilon-D_{X}^{3}F_{T}(X_{2})\Xi\Upsilon||\leq c_{T}||X_{1}-X_{2}||^{\delta}||\Xi||\,||\Upsilon||,\\
||D_{X}^{3}F(X)\Xi\Upsilon||\leq c||\Xi||\,||\Upsilon||,\:||D_{X}^{3}F_{T}(X)\Xi\Upsilon||\leq c_{T}||\Xi||\,||\Upsilon||.
\end{array}\label{eq:6-2}
\end{equation}
We make the measurability assumption 
\begin{equation}
D_{X}^{3}F(X),\,D_{X}^{3}F_{T}(X)\,\text{are }\sigma(X)\,\text{measurable}\label{eq:6-3}
\end{equation}
This must be interpreted as for the second derivative, cf.~(\ref{eq:4-102}).
The assumption means 
\begin{equation}
D_{X}^{3}F(X)\Xi\Upsilon=A_{X}(X)\Xi\Upsilon+B_{X\Xi}(X)\Upsilon+C_{X\Upsilon}(X)\Xi+D_{X\Xi\Upsilon}\label{eq:6-41}
\end{equation}
in which the maps $x\rightarrow A_{X}(x),\,x\rightarrow B_{X\Xi}(x),\,x\rightarrow C_{X\Upsilon}(x),$
$x\rightarrow D_{X\Xi\Upsilon}$ are deterministic from $\mathbb{R}^{n}$ to
 $\mathcal{L}(\mathbb{R}^{n};\mathcal{L}(\mathbb{R}^{n};\mathbb{R}^{n})),\,\mathcal{L}(\mathbb{R}^{n};\mathbb{R}^{n}),\,\mathcal{L}(\mathbb{R}^{n};\mathbb{R}^{n})$
and $\mathbb{R}^{n},$ respectively.
The maps $\Xi\rightarrow B_{X\Xi}(x),\,\Upsilon\rightarrow C_{X\Upsilon}(x)$
are linear and $\Xi,\Upsilon\rightarrow D_{X\Xi\Upsilon}$ is bilinear.
It follows that $D_{X}^{3}F(X)\Xi\Upsilon$ is $\sigma(X,\Xi,\Upsilon)$
measurable.

\subsection{THE MASTER EQUATION}

Our objective is to prove the following 
\begin{theorem}
\label{theo6-1} We make the assumptions of Proposition \ref{prop:5-1}
and \eqref{eq:6-2}, \eqref{eq:6-3}. Then $\mathcal{U}(X,t)$ is
differentiable in $t,$ and $((D_{X}\mathcal{U}(X,t)\sigma N,\sigma N))$
is differentiable in $X$. Moreover it is the solution of the following \emph{Master Equation}:
\begin{equation}
\begin{array}{c}
\dfrac{\partial\mathcal{U}}{\partial t}+\dfrac{1}{2}D_{X}((D_{X}\mathcal{U}(X,t)\sigma N,\sigma N))-\dfrac{1}{\lambda}D_{X}\mathcal{U}(X,t)\mathcal{\,U}(X,t)+D_{X}F(X)=0\\
\mathcal{U}(X,T)=D_{X}F_{T}(X)
\end{array}\label{eq:6-4}
\end{equation}
We have 
\begin{equation}
||\dfrac{\partial\mathcal{U}}{\partial t}(X,t)||\leq C(1+||X||)\label{eq:6-5}
\end{equation}
 and $\dfrac{\partial\mathcal{U}}{\partial t}(X,t)$ is $\sigma(X)$
measurable. Also $\dfrac{\partial\mathcal{U}}{\partial t}$ is H\"older
in $X.$
\end{theorem}
\begin{proof}
Consider the various terms in equation (\ref{eq:5-10}). Consider
first $||D_{X}V(X,t)||^{2}.$ We first check the G\^ateaux derivative 
\[
\dfrac{d}{d\theta}||D_{X}V(X+\theta\tilde{X},t)||^{2}|_{\theta=0}=2((D_{X}V(X,t),D_{X}^{2}V(X,t)\tilde{X}))
\]
for any $\tilde{X}$ $\mathcal{F}^{t}$ measurable. Since $D_{X}^{2}V(X,t)$
is self adjoint we have $((D_{X}V(X,t),D_{X}^{2}V(X,t)\tilde{X}))=((D_{X}^{2}V(X,t)D_{X}V(X,t),\tilde{X}))$
and thus we obtain immediately 
\begin{equation}
D_{X}||\mathcal{U}(X,t)||^{2}=2D_{X}\mathcal{U}(X,t)\mathcal{\,U}(X,t)\label{eq:6-6}
\end{equation}
 We consider next $((D_{X}\mathcal{U}(X,t)\sigma N,\sigma N)).$ To
study this function we need a formula. Recalling (\ref{eq:5-2}) we
write successively 
\begin{equation}
\begin{array}{c}
\mathcal{Y}^{Xt}(s)  =\mathcal{\sigma}N-\dfrac{1}{\lambda}\int_{t}^{s}\mathcal{Z}^{Xt}(\tau)d\tau\\
\mathcal{Z}^{Xt}(s)=\mathbb{E}[D_{X}^{2}F_{T}(Y_{Xt}(T))\mathcal{Y}^{Xt}(T)+\int_{s}^{T}D_{X}^{2}F(Y_{Xt}(\tau))\mathcal{Y}^{Xt}(\tau)d\tau|\mathcal{F}^{s}\cup\sigma(N)]
\end{array}
\label{eq:6-7}
\end{equation}
in which we emphasize the dependence in the pair $X,t.$ Also 
\begin{equation}
\begin{array}{c}
Y_{Xt}(s)=X-\dfrac{1}{\lambda}\int_{t}^{s}Z_{Xt}(\tau)d\tau+\sigma(w(s)-w(t))\\
Z_{Xt}(s)=\mathbb{E}[D_{X}F_{T}(Y_{Xt}(T))+\int_{s}^{T}D_{X}F(Y_{Xt}(\tau))d\tau|\mathcal{F}^{s}]
\end{array}\label{eq:6-8}
\end{equation}
 Then the function $\Phi(X,t)=((D_{X}\mathcal{U}(X,t)\sigma N,\sigma N))$
is given by the formula 
\begin{equation}
\Phi(X,t)=\dfrac{1}{\lambda}\int_{t}^{T}||\mathcal{Z}^{Xt}(s)||^{2}ds+\int_{t}^{T}((D_{X}^{2}F(Y_{Xt}(s))\mathcal{Y}^{Xt}(s),\mathcal{Y}^{Xt}(s)))ds+((D_{X}^{2}F_{T}(Y_{Xt}(T))\mathcal{Y}^{Xt}(T),\mathcal{Y}^{Xt}(T))),\label{eq:6-9}
\end{equation}
cf.~\eqref{eq:5-5}.
 Thanks to the assumptions $(\ref{eq:6-2}),$ (\ref{eq:6-3}) we can
differentiate in $X.$ Let $\tilde{X}$ be $\mathcal{F}^{t}$ measurable.
We define the pair $\tilde{\mathcal{Y}}(s),\tilde{\mathcal{Z}}(s)$
by 
\begin{equation}
\begin{array}{c}
\tilde{\mathcal{Y}}(s)=\tilde{X}-\dfrac{1}{\lambda}\int_{t}^{s}\tilde{\mathcal{Z}(}\tau)d\tau\\
\tilde{\mathcal{Z}(}s)=\mathbb{E}[D_{X}^{2}F_{T}(Y(T))\tilde{\mathcal{Y}}(T)+\int_{s}^{T}D_{X}^{2}F(Y(\tau))\tilde{\mathcal{Y}}(\tau)d\tau|\mathcal{F}^{s}]
\end{array}\label{eq:6-10}
\end{equation}
and the pair $\mathcal{Y}'(s),$$\mathcal{Z}'(s)$ by the equations
\begin{multline}
\label{eq:6-11}\mathcal{Y}'(s)=-\dfrac{1}{\lambda}\int_{t}^{s}\mathcal{Z}'(\tau)d\tau,\\
\mathcal{Z}'(s)=\mathbb{E}[D_{X}^{3}F_{T}(Y(T))\tilde{\mathcal{Y}}(T)\mathcal{Y}(T)+D_{X}^{2}F_{T}(Y(T))\mathcal{Y}'(T)\\
+\int_{s}^{T}(D_{X}^{3}F(Y(\tau))\tilde{\mathcal{Y}}(\tau)\mathcal{Y}(\tau)+D_{X}^{2}F(Y(\tau))\mathcal{Y}'(\tau))d\tau|\mathcal{F}^{s}\cup\sigma(N)]
\end{multline}
 and we have 
\begin{multline*}
\dfrac{d}{d\theta}\Phi(X+\theta\tilde{X},t)|_{\theta=0}=\dfrac{2}{\lambda}\int_{t}^{T}((\mathcal{Z}(s),\mathcal{Z}'(s)))ds+\int_{t}^{T}((D_{X}^{3}F(Y(s))\tilde{\mathcal{Y}}(s)\mathcal{Y}(s),\mathcal{Y}(s)))ds\\
+2\int_{t}^{T}((D_{X}^{2}F(Y(s))\mathcal{Y}'(s),\mathcal{Y}(s)\,))ds+((D_{X}^{3}F_{T}(Y(T))\tilde{\mathcal{Y}}(T)\mathcal{Y}(T),\mathcal{Y}(T)))+2((D_{X}^{2}F(Y(T))\mathcal{Y}'(T),\mathcal{Y}(T)\,)).
\end{multline*}
We deduce that 
\begin{equation}
\begin{array}{c}
\sup_{t\leq s\leq T}||\mathcal{Y}(s)||,\:\sup_{t\leq s\leq T}||\mathcal{Z}(s)||\leq C,\\
\sup_{t\leq s\leq T}||\mathcal{\tilde{Y}}(s)||,\:\sup_{t\leq s\leq T}||\mathcal{\tilde{Z}}(s)||\leq C||\tilde{X}||,\\
\sup_{t\leq s\leq T}||\mathcal{Y}(s)||,\:\sup_{t\leq s\leq T}||\mathcal{Z}(s)||\leq C||\tilde{X}||.
\end{array}\label{eq:6-12}
\end{equation}
Therefore the map $\tilde{X}\rightarrow\dfrac{d}{d\theta}\Phi(X+\theta\tilde{X},t)|_{\theta=0}$
is linear and $|\dfrac{d}{d\theta}\Phi(X+\theta\tilde{X},t)|_{\theta=0}|\leq C||\tilde{X}||.$
Therefore $\Phi(X,t)$ is differentiable in $X$ and 
\begin{multline}
((D_{X}\Phi(X,t),\tilde{X}))=\dfrac{2}{\lambda}\int_{t}^{T}((\mathcal{Z}(s),\mathcal{Z}'(s)))ds+\int_{t}^{T}((D_{X}^{3}F(Y(s))\tilde{\mathcal{Y}}(s)\mathcal{Y}(s),\mathcal{Y}(s)))ds\\
+2\int_{t}^{T}((D_{X}^{2}F(Y(s))\mathcal{Y}'(s),\mathcal{Y}(s)\,))ds+((D_{X}^{3}F_{T}(Y(T))\tilde{\mathcal{Y}}(T)\mathcal{Y}(T),\mathcal{Y}(T)))+2((D_{X}^{2}F(Y(T))\mathcal{Y}'(T),\mathcal{Y}(T)\,)).\label{eq:6-13}
\end{multline}
We can slightly rearrange this formula. We note that $\mathcal{Z}'(s)=\mathbb{E}[\varUpsilon'(s)|\mathcal{F}^{s}\cup\sigma(N)]$
with 
\begin{multline*}
\varUpsilon'(s)=D_{X}^{3}F_{T}(Y(T))\tilde{\mathcal{Y}}(T)\mathcal{Y}(T)+D_{X}^{2}F_{T}(Y(T))\mathcal{Y}'(T)\\
+\int_{s}^{T}(D_{X}^{3}F(Y(\tau))\tilde{\mathcal{Y}}(\tau)\mathcal{Y}(\tau)+D_{X}^{2}F(Y(\tau))\mathcal{Y}'(\tau))d\tau
\end{multline*}
and $\dfrac{1}{\lambda}\int_{t}^{T}((\mathcal{Z}(s),\mathcal{Z}'(s)))ds=-\int_{t}^{T}((\dfrac{d\mathcal{Y}(s)}{ds},\varUpsilon'(s)\,)).$
Performing integration by parts, we finally obtain the formula
\begin{multline}
((D_{X}((D_{X}\mathcal{U}(X,t)\sigma N,\sigma N)),\tilde{X}))=-((\mathcal{Y}(T),D_{X}^{3}F_{T}(Y(T))\tilde{\mathcal{Y}}(T)\mathcal{Y}(T)\,)) \label{eq:6-14}\\
-\int_{t}^{T}((\mathcal{Y}(s),D_{X}^{3}F_{T}(Y(s))\tilde{\mathcal{Y}}(s)\mathcal{Y}(s)\,))ds\\
+ 2((\sigma N,D_{X}^{3}F_{T}(Y(T))\tilde{\mathcal{Y}}(T)\mathcal{Y}(T)+D_{X}^{2}F_{T}(Y(T))\mathcal{Y}'(T)+\int_{s}^{T}(D_{X}^{3}F(Y(\tau))\tilde{\mathcal{Y}}(\tau)\mathcal{Y}(\tau)+D_{X}^{2}F(Y(\tau))\mathcal{Y}'(\tau))d\tau)).
\end{multline}
Going back
to the Bellman equation \eqref{eq:5-10}, rewritten as 
\begin{equation*}
\dfrac{\partial V}{\partial t}+\dfrac{1}{2}((D_{X}\mathcal{U}(X,t)\sigma N,\sigma N))-\dfrac{1}{2\lambda}||\mathcal{U}(X,t)||^{2}+F(X)=0,
\end{equation*}
we deduce from \eqref{eq:6-6} and \eqref{eq:6-14} that $D_{X}V(X,t)$ is differentiable in $t,$ and equation \eqref{eq:6-4}
holds. We have $||D_{X}((D_{X}\mathcal{U}(X,t)\sigma N,\sigma N))||\leq C.$
From Proposition \ref{prop:4-1} we have $||D_{X}^{2}V(X,t)||\leq C.$
Since
$||D_{X}V(X,t)||\leq C(1+||X||)$, the estimate (\ref{eq:6-5}) follows. The fact that $\dfrac{\partial\mathcal{U}}{\partial t}$ is
H\"older in $X$ follows from the assumption (\ref{eq:6-2}) and formulas
(\ref{eq:6-6}) and (\ref{eq:6-14}), with tedious but straightforward calculations.
The proof has been completed. 
\end{proof}

\section{MEAN FIELD TYPE CONTROL PROBLEMS}

\subsection{FUNCTIONALS}

We apply the preceding results to the mean field type control problem
situation in which 
\begin{align}
F(X) & =F(\mathcal{L}_{X})=F(m)=\int_{\mathbb{R}^{n}}f(x,m)m(x)dx\label{eq:7-1}\\
F_{T}(X) & =F_{T}(\mathcal{L}_{X})=F_{T}(m)=\int_{\mathbb{R}^{n}}h(x,m)m(x)dx\nonumber 
\end{align}
We consider random variables with densities $m(x)$ belonging to $L^{2}(\mathbb{R}^{n}).$ We
assume differentiability in $m$ as follows $\dfrac{\partial F(m)}{\partial m}(x),$$\dfrac{\partial^{2}F(m)}{\partial m^{2}}(x_{1},x_{2})$,
$\dfrac{\partial^{3}F(m)}{\partial m^{3}}(x_{1},x_{2})(x)$ such that
the following expansion is valid 
\begin{multline}
F(m+\tilde{m})=F(m)+\int_{\mathbb{R}^{n}}\dfrac{\partial F(m)}{\partial m}(x)\tilde{m}(x)dx+\dfrac{1}{2}\int_{\mathbb{R}^{n}}\int_{\mathbb{R}^{n}}\dfrac{\partial^{2}F(m)}{\partial m^{2}}(x_{1},x_{2})\tilde{m}(x_{1})\tilde{m}(x_{2})dx_{1}dx_{2}\label{eq:7-2} \\
+\int_{0}^{1}\int_{0}^{1}\theta^{2}\mu\int_{\mathbb{R}^{n}}\int_{\mathbb{R}^{n}}\int_{\mathbb{R}^{n}}\dfrac{\partial^{3}F(m+\theta\mu\tilde{m})}{\partial m^{3}}(x_{1},x_{2})(x)\tilde{m}(x_{1})\tilde{m}(x_{2})\tilde{m}(x)dx_{1}dx_{2}dx.
\end{multline}
 We have analogous formulas for $F_{T}.$ In \cite{bensoussan2017interpretation} it is shown
that ($m$ represents the probability density of $X)$
\begin{equation}
D_{X}F(X)=D_{x}\dfrac{\partial F(m)}{\partial m}(X),\label{eq:7-3}
\end{equation}
\begin{equation}
D_{X}^{2}F(X)Z=D_{x}^{2}\dfrac{\partial F(m)}{\partial m}(X)Z+\mathbb{E}_{\tilde{X}\tilde{Z}}[D_{x}D_{\tilde{x}}\dfrac{\partial^{2}F(m)}{\partial m^{2}}(X,\tilde{X})\tilde{Z}]\label{eq:7-4}
\end{equation}
 in which $\tilde{X}, \tilde{Z}$ is an independent copy of $X, Z,$
and $\mathbb{E}_{\tilde{X}\tilde{Z}}$ is the expectation with respect to the
variables $\tilde{X}, \tilde{Z}$ leaving $X$ fixed. A more elaborate
formula can be given for $D_{X}^{3}F(X)\Xi\Upsilon.$ We state it
formally:
\begin{multline}
D_{X}^{3}F(X)\Xi\Upsilon=D_{x}\Upsilon^{*}D_{x}^{2}\dfrac{\partial F(m)}{\partial m}(X)\Xi \label{eq:7-5}\\
+D_{x}\mathbb{E}_{\tilde{X}\tilde{\Xi}}\tilde{\Xi}^{*}D_{\tilde{x}}D_{x}\dfrac{\partial^{2}F(m)}{\partial m^{2}}(X,\tilde{X})\Upsilon+D_{x}\mathbb{E}_{\tilde{X}\tilde{\Upsilon}}\tilde{\Upsilon}^{*}D_{\tilde{x}}D_{x}\dfrac{\partial^{2}F(m)}{\partial m^{2}}(X,\tilde{X})\Xi \\
+D_{x}\mathbb{E}_{\tilde{X}_{1}\tilde{\Xi}}\mathbb{E}_{\tilde{X}_{2}\tilde{\Upsilon}}\tilde{\Upsilon}^{*}D_{\tilde{x}_{1}}D_{\tilde{x}_{2}}\dfrac{\partial^{3}F(m)}{\partial m^{3}}(\tilde{X}_{1},\tilde{X}_{2})(X)\Xi.
\end{multline}
 With these formulas, one can easily check the measurability properties
(\ref{eq:2-4}), (\ref{eq:4-102}), (\ref{eq:6-3}).

\subsection{INTERPRETATION }

Let us write 
\begin{equation}
F(x,m)=\dfrac{\partial F(m)}{\partial m}(x)=f(x,m)+\int_{\mathbb{R}^{n}}\dfrac{\partial f(\xi,m)}{\partial m}(x)m(\xi)d\xi,\label{eq:7-6}
\end{equation}
\[
F_{T}(x,m)=\dfrac{\partial F_{T}(m)}{\partial m}(x)=h(x,m)+\int_{\mathbb{R}^{n}}\dfrac{\partial h(\xi,m)}{\partial m}(x)m(\xi)d\xi,
\]
 so that $D_{X}F(X)=D_{x}F(X,\mathcal{L}_{X}),$ $D_{X}F_{T}(X)=D_{x}F_{T}(X,\mathcal{L}_{X})$.
The system (\ref{eq:3-1}) becomes 
\begin{equation}
Y(s)=X-\dfrac{1}{\lambda}\int_{t}^{s}Z(\tau)d\tau+\sigma(w(s)-w(t))\label{eq:7-7}
\end{equation}
\[
Z(s)=\mathbb{E}[D_{x}F_{T}(Y(T),\mathcal{L}_{Y(T)})+\int_{s}^{T}D_{x}F(Y(\tau),\mathcal{L}_{Y(\tau)})d\tau\:|\mathcal{W}_{X,t}^{s}]
\]

To relate this system to a mean field type control problem, we associate
to (\ref{eq:7-7}) another system as follows. We consider a triple
$x,m,t$ where $x\in \mathbb{R}^{n},$ $m$ is a probability density on $\mathbb{R}^{n}$. Let us consider a stochastic process $y_{xmt}(s)$
with values in $\mathbb{R}^{n}$, which is adapted to $\mathcal{W}_{t}^{s}=\sigma(w(\tau)-w(t),t\leq\tau\leq s)$
and such that $y_{xmt}(t)=x.$ We write $y_{mt}(s)(x)=y_{xmt}(s)$
and call $y_{mt}(s)(.)\#m$ the image measure of $m$ by the random
function $y_{mt}(s)(.),$ called also the push forward probability
measure of $m$. It is simply the probability distribution of $y_{\xi mt}(s)$
when $\xi$ is random variable on $\mathbb{R}^{n}$, with probability distribution
$m,$ independent of $\mathcal{W}_{t}^{s}.$ We then define the system,
in which $y_{xmt}(s),z_{xmt}(s)$ are two processes with values in
$\mathbb{R}^{n}$ adapted to $\mathcal{W}_{t}^{s}$ such that 
\begin{equation}
\begin{array}{c}
y_{xmt}(s)=x-\dfrac{1}{\lambda}\int_{t}^{s}z_{xmt}(\tau)d\tau+\sigma(w(s)-w(t)),\\
z_{xmt}(s)=\mathbb{E}[D_{x}F_{T}(y_{xmt}(T),y_{mt}(T)(.)\#m)+\int_{s}^{T}D_{x}F(y_{xmt}(\tau),y_{mt}(\tau)(.)\#m)d\tau|\mathcal{W}_{t}^{s}].
\end{array}\label{eq:7-8}
\end{equation}

\subsection{MEAN FIELD TYPE CONTROL PROBLEM}

We first note that if we take $x=X$ and $m=\mathcal{L}_{X}$ then
$y_{X\mathcal{L}_{X}t}(s)=Y(s)$ and $z_{X\mathcal{L}_{X}t}(s)=Z(s),$
since clearly $y_{\mathcal{L}_{X}t}(s)(.)\#\mathcal{L}_{X}=\mathcal{L}_{y_{X\mathcal{L}_{X}t}(s)}.$
We now interpret (\ref{eq:7-8}) as a necessary and sufficient optimality condition
of a control problem. The controls are stochastic processes with values
in $\mathbb{R}^{n}$ adapted to $\mathcal{W}_{t}^{s}=\sigma(w(\tau)-w(t),\,t\leq\tau\leq s)$
and dependent on initial conditions $x,t.$ We denote a control by
$v_{xt}(s).$ We leave $x$ as an index rather than an argument, to
emphasize we are not considering deterministic feedback controls;
it is really an initial condition parameter. We assume $\mathbb{E}\int_{t}^{T}\int_{\mathbb{R}^{n}}|v_{xt}(s)|^{2}m(x)dxds<+\infty,$
in which, as above, $m$ is a probability density. So the space of
controls is the Hilbert space $L_{\mathcal{W}_{t}^{s}}^{2}(t,T;L^{2}(\Omega,\mathcal{A},\mathbb{P};L_{m}^{2}(\mathbb{R}^{n};\mathbb{R}^{n}))).$
We define the state $x_{x,t}(s)$ by 
\begin{equation}
x_{x,t}(s)=x+\int_{t}^{s}v_{xt}(\tau)d\tau+\sigma(w(s)-w(t)). \label{eq:7-9}
\end{equation}
We use the notation $x_{x,t}(s;v(\cdot))$ to emphasize the dependence
in the control. We write, as above $x_{t}(s;v(\cdot))(x)=x_{x,t}(s;v(\cdot))$
to focus on the function $x\rightarrow x_{x,t}(s;v(\cdot)).$ We shall
use the push forward probability $x_{t}(s;v(\cdot))(.)\#m.$ We then define
the cost function 
\begin{multline}
J_{mt}(v(\cdot))=\dfrac{\lambda}{2}\mathbb{E}\int_{t}^{T}\int_{\mathbb{R}^{n}}|v_{xt}(s)|^{2}m(x)dxds+\mathbb{E}\int_{t}^{T}\int_{\mathbb{R}^{n}}f(x_{x,t}(s;v(\cdot)),x_{t}(s;v(\cdot))(.)\#m)m(x)dxds \label{eq:7-10} \\
+\mathbb{E}\int_{\mathbb{R}^{n}}h(x_{x,t}(T;v(\cdot)),x_{t}(T;v(\cdot))(.)\#m)m(x)dx.
\end{multline}
If we compare \eqref{eq:7-10} with \eqref{eq:2-52} it is easy to
convince ourselves that they are identical, provided $m=\mathcal{L}_{X}.$
Indeed $v_{Xt}(s)$ is adapted to $\mathcal{W}_{X,t}^{s}$ that we
write $v(s)$ and $\mathbb{E}\int_{t}^{T}\int_{\mathbb{R}^{n}}|v_{xt}(s)|^{2}m(x)dxds=\int_{t}^{T}||v(s)||^{2}ds.$
Moreover $x_{X,t}(s;v(\cdot))=X_{X,t}(s;v(\cdot))$ and $x_{t}(s;v(\cdot))(.)\#\mathcal{L}_{X}=\mathcal{L}_{X_{X,t}(s;v(\cdot))}.$ Therefore 
\begin{equation}
\mathbb{E}\int_{t}^{T}\int_{\mathbb{R}^{n}}f(x_{x,t}(s;v(\cdot)),x_{t}(s;v(\cdot))(.)\#m)m(x)dxds=\int_{t}^{T}Ef(X(s);\mathcal{L}_{X(s)})ds=\int_{t}^{T}F(X(s))ds\label{eq:7-101}
\end{equation}
\[
\mathbb{E}\int_{\mathbb{R}^{n}}h(x_{x,t}(T;v(\cdot)),x_{t}(T;v(\cdot))(.)\#m)m(x)dx=Eh(X(T);\mathcal{L}_{X(T)})=F_{T}(X(T))
\]
 and thus $J_{mt}(v(\cdot))=J_{X,t}(v(\cdot)).$ Conversely, we can write any
control adapted to $\mathcal{W}_{X,t}^{s}$ as $v_{Xt}(s)$ where $v_{xt}(s)$ is
adapted to $\mathcal{W}_{t}^{s}.$ Therefore the value function $V(X,t)$
depends only on the law of $X,$ and 
\begin{equation}
V(X,t)=V(m,t)=\inf_{v(\cdot)}J_{mt}(v(\cdot)).\label{eq:7-11}
\end{equation}
We can then compute the G\^ateaux derivative of $J_{mt}(v(\cdot))$ in the
Hilbert space\\ $L_{\mathcal{W}_{t}^{s}}^{2}(t,T;L^{2}(\Omega,\mathcal{A},\mathbb{P};L_{m}^{2}(\mathbb{R}^{n};\mathbb{R}^{n}))).$
We have the following result which mimics that of Proposition \ref{prop:2-1} .
\begin{prop}
\label{prop:7-1}We assume that the funtions $F(X)$ and $F_{T}(X)$
defined by (\ref{eq:7-1}) satisfy (\ref{eq:2-1}). We have 
\begin{multline}
D_{v(\cdot);xmt}(s)=\lambda v_{xt}(s)+\mathbb{E}[D_{x}F_{T}(x_{x,t}(T;v(\cdot));x_{t}(T;v(\cdot))(.)\#m) \label{eq:7-12}\\
+\int_{s}^{T}D_{x}F(x_{x,t}(\tau;v(\cdot));x_{t}(\tau;v(\cdot))(.)\#m)d\tau\,|\mathcal{W}_{t}^{s}].
\end{multline}
\end{prop}

\begin{proof}
We only sketch it. We recall that functions $F(x,m)$ and $F_{T}(x,m)$
are defined by (\ref{eq:7-6}). It is of course useful to connect
the calculation with that of Proposition \ref{prop:2-1}. The important
observation is the following: If we modify $v(\cdot)$ into $v(\cdot)+\theta\tilde{v}(.)$,
then the state $x_{x,t}(s;v(\cdot))$ is changed into $x_{x,t}(s;v(\cdot))+\theta\int_{t}^{s}\tilde{v}_{xt}(\tau)d\tau.$
Recalling the notation $x_{X,t}(s;v(\cdot))=X_{X,t}(s;v(\cdot))=X(s),$ this
amounts to changing $X(s)$ into $X(s)+\theta\int_{t}^{s}\tilde{v}(\tau)d\tau,$
where $\tilde{v}(s)=\tilde{v}_{Xt}(s)$ which is adapted $\mathcal{W}_{X,t}^{s}.$
From (\ref{eq:7-101}) it follows that $\int_{t}^{T}F(X(s))ds$ is
changed into $\int_{t}^{T}F(X(s)+\theta\int_{t}^{s}\tilde{v}(\tau)d\tau)ds$.
From the differentiability of $F(X)$ we see that 
\begin{multline*}
\dfrac{\int_{t}^{T}F(X(s)+\theta\int_{t}^{s}\tilde{v}(\tau)d\tau)ds-\int_{t}^{T}F(X(s)ds}{\theta}\rightarrow\int_{t}^{T}\mathbb{E}\,D_{x}F(X(s),\mathcal{L}_{X(s)}).\int_{t}^{s}\tilde{v}(\tau)d\tau)ds\\
=\mathbb{E}\int_{t}^{T}\int_{\mathbb{R}^{n}}D_{x}F(x_{x,t}(s;v(\cdot)),x_{t}(s;v(\cdot))(.)\#m).\int_{t}^{s}\tilde{v}_{xt}(\tau)d\tau\,m(x)dxds
\end{multline*}
 and similar results for $F_{T}.$ Performing rearrangements as in
the proof of Proposition \ref{prop:2-1}, we obtain formula
(\ref{eq:7-12}). 
\end{proof}
From formula (\ref{eq:7-12}), it is immediate that the optimal control
is $u_{xmt}(s)=-\dfrac{1}{\lambda}z_{xmt}(s),$ in which the pair $y_{xmt}(s),$$z_{xmt}(s)$
is the unique solution of (\ref{eq:7-8}).

\subsection{SYSTEM OF HJB-FP EQUATIONS }

We first define the probability $m_{mt}(s)=y_{mt}(s)(.)\#m$. We then
define an ordinary stochastic control problem, with $m,t$ as parameters
and $x,s$ as initial conditions. Let $v(\tau)$ be adapted to $\mathcal{W}_{s}^{\tau}=\sigma(w(\theta)-w(s),s\leq\theta\leq\tau).$
The state $x(\tau)$ is given by 
\begin{equation}
x(\tau)=x+\int_{s}^{\tau}v(\theta)d\theta+\sigma(w(\tau)-w(s))\label{eq:7-13}
\end{equation}
and we define the cost function by 
\begin{equation}
J_{mt}(x,s;v(\cdot))=\dfrac{\lambda}{2}\mathbb{E}\int_{s}^{T}|v(\tau)|^{2}d\tau+\mathbb{E}\int_{s}^{T}F(x(\tau),m_{mt}(\tau))d\tau+EF_{T}(x(T),m_{mt}(T)).\label{eq:7-14}
\end{equation}
In this functional $m_{mt}(\tau)$ is frozen . This is why the problem
(\ref{eq:7-13}), (\ref{eq:7-14}) is a standard stochastic control
problem. Writing the necessary and sufficient condition of optimality,
it is easy to check that the optimal control is $-\dfrac{1}{\lambda}z_{mt}(x,s;\tau)$
and the optimal trajectory is $y_{mt}(x,s;\tau),$ where $y_{mt}(x,s;\tau),$
$z_{mt}(x,s;\tau)$ are solutions of 
\begin{equation}
y_{mt}(x,s;\tau)=x-\dfrac{1}{\lambda}\int_{s}^{\tau}z_{mt}(x,s;\theta)d\theta+\sigma(w(\tau)-w(s)),\label{eq:7-15}
\end{equation}
\[
z_{mt}(x,s;\tau)=\mathbb{E}[D_{x}F_{T}(y_{mt}(x,s;T),m_{mt}(T))+\int_{\tau}^{T}D_{x}F(y_{mt}(x,s;\theta),m_{mt}(\theta))d\theta|\mathcal{W}_{s}^{\tau}].
\]
 Comparing with \eqref{eq:7-8} we see that $y_{xmt}(\tau)=y_{mt}(x,t;\tau)$
and $z_{xmt}(\tau)=z_{mt}(x,t;\tau).$ We next define 
\begin{equation}
u_{mt}(x,s)=\inf_{v(\cdot)}J_{mt}(x,s;v(\cdot))\label{eq:7-16}
\end{equation}
and thus we get the formula 
\begin{equation}
u_{mt}(x,s)=\dfrac{1}{2\lambda}\mathbb{E}\int_{s}^{T}|z_{mt}(x,s;\tau)|^{2}d\tau+\mathbb{E}\int_{s}^{T}F(y_{mt}(x,s;\tau),m_{mt}(\tau))d\tau+EF_{T}(y_{mt}(x,s;T),m_{mt}(T)).\label{eq:7-17}
\end{equation}
From \eqref{eq:7-16} and a simple application of the envelope theorem
we can write 
\begin{align}
D_{x}u_{mt}(x,s) & =\mathbb{E}[D_{x}F_{T}(y_{mt}(x,s;T),m_{mt}(T))+\int_{s}^{T}D_{x}F(y_{mt}(x,s;\theta),m_{mt}(\theta))d\theta]\label{eq:7-18}\\
 & =z_{mt}(x,s;s).\nonumber 
\end{align}
 Now we observe that $y_{mt}(y_{xmt}(s),s;\tau)=y_{xmt}(\tau)$. Also,
from \eqref{eq:7-18},
\begin{align*}
z_{mt}(y_{xmt}(s),s;s)&=\mathbb{E}[D_{x}F_{T}(y_{mt}(x,s;T),m_{mt}(T))+\int_{s}^{T}D_{x}F(y_{mt}(x,s;\theta),m_{mt}(\theta))d\theta]|_{x=y_{xmt}(s)}\\
 &=\mathbb{E}[D_{x}F_{T}(y_{xmt}(T),m_{mt}(T))+\int_{s}^{T}D_{x}F(y_{xmt}(\tau),m_{mt}(\tau))d\tau|\mathcal{W}_{t}^{s}]\\
 &=z_{xmt}(s)
\end{align*}
and thus we have obtained 
\begin{equation}
z_{xmt}(s)=D_{x}u_{mt}(y_{xmt}(s),s).\label{eq:7-19}
\end{equation}
It follows that the evolution of $y_{xmt}(s)$ is given by 
\begin{equation}
y_{xmt}(s)=x-\dfrac{1}{\lambda}\int_{t}^{s}D_{x}u_{mt}(y_{xmt}(\tau),\tau)d\tau+\sigma(w(s)-w(t))\label{eq:7-20}
\end{equation}
cf.~\eqref{eq:7-8}.
But then the push-forward probability density $m_{mt}(s)(.)=y_{mt}(s)(.)\#m$
is the solution of Fokker Planck equation
\begin{align}
\dfrac{\partial m}{\partial s}+Am-\dfrac{1}{\lambda}\text{div }(mDu) & =0\label{eq:7-21}\\
m(x,t)=m(x)\nonumber 
\end{align}
 where we have dropped the indices $m,t$ to simplify notation. On
the other hand, from the definition of $u_{mt}(x,s)$ (see \eqref{eq:7-16}),
and assuming appropriate regularity, it is the solution of the HJB equation 
\begin{align}
-\dfrac{\partial u}{\partial s}+Au+\dfrac{1}{2\lambda}|Du|^{2} & =F(x,m(s))\label{eq:7-22}\\
u(x,T) & =F_{T}(x,m(T))\nonumber 
\end{align}
 In (\ref{eq:7-21}), (\ref{eq:7-22}) the operator $A$ is defined
by $A\varphi(x)=-\dfrac{1}{2}\text{tr }(aD^{2}\varphi)(x),$ where
$a=\sigma\sigma^{*}.$ 

We obtain the classical system of HJB-FP equations of mean field type
control.

\subsection{INTERPRETATION OF $u_{mt}(x,t)$ }

In \cite{bensoussan2015control} in the case $\sigma=0,$ deterministic case, we have
proven that 
\begin{equation}
\dfrac{\partial V(m,t)}{\partial m}(x)=u_{mt}(x,t)\label{eq:7-23}
\end{equation}
Let us give a formal proof in the present case of this result. We
use 
\[
Z_{Xt}(t)=D_{X}V(X,t)
\]
 But from the above discussion $Z_{Xt}(t)=z_{X\mathcal{L}_{X}t}(t)=D_{x}u_{\mathcal{L}_{X}t}(X,t).$
On the other hand, since $V(X,t)=V(\mathcal{L}_{X},t),$ we have
also $D_{X}V(X,t)=D_{x}\dfrac{\partial V(\mathcal{L}_{X},t)}{\partial m}(X).$
Therefore
$D_{x}u_{\mathcal{L}_{X}t}(X,t)=D_{x}\dfrac{\partial V(\mathcal{L}_{X},t)}{\partial m}(X),$
which means 
\[
D_{x}u_{m}(x,t)=D_{x}\dfrac{\partial V(m,t)}{\partial m}(x)
\]
 from which we infer (\ref{eq:7-23}).

\newpage
\setcounter{page}{1}

\appendix
\section{Proofs of the statements in Section \ref{sec:FORMALISM}}
\subsection{Proof of Lemma \ref{Lemma2_2}}
\begin{proof}
As before, we introduce a pair of independent copy of $(X_{m},X_{m'})$, namely $(\tilde{X}_{m}, \tilde{X}_{m'})$, the relation (\ref{eq:2-112}) is now equivalent to 
\small{
\begin{align} \label{eq:2-114}
&\dfrac{d^{2}}{dt^{2}}u(m+t(m'-m))\\
& \quad =\mathbb{E}\left(\dfrac{\partial^{2}u}{\partial m^{2}}(m+t(m'-m))(X_{m'},\tilde{X}_{m'})\right)+\mathbb{E}\left(\dfrac{\partial^{2}u}{\partial m^{2}}(m+t(m'-m))(X_{m},\tilde{X}_{m})\right) \nonumber \\
& \quad \quad -\mathbb{E}\left(\dfrac{\partial^{2}u}{\partial m^{2}}(m+t(m'-m))(X_{m},\tilde{X}_{m}')\right)-\mathbb{E}\left(\dfrac{\partial^{2}u}{\partial m^{2}}(m+t(m'-m))(X_{m'},\tilde{X}_{m})\right),  \nonumber
\end{align}}
a simple application of mean value theorem gives that :
\small{\begin{align}\label{eq:2-116}
&\dfrac{d^{2}}{dt^{2}}u(m+t(m'-m)) \\
& \quad =\mathbb{E}\left(\int_{0}^{1}\int_{0}^{1}D_{\xi}D_{\eta}\dfrac{\partial^{2}u}{\partial m^{2}}(m+t(m'-m))(X_{m}+\alpha(X_{m'}-X_{m}),\tilde{X}_{m}+\beta(\tilde{X}_{m'}-\tilde{X}_{m}))(\tilde{X}_{m'}-\tilde{X}_{m}) \cdot (X_{m'}-X_{m}) d\alpha d\beta \right). \nonumber 
\end{align}}
Let $f(t)=u(m+t(m'-t))$, under the assumptions, $f$ is then $C^{2}$ in t. Therefore, we have
\[
f(1)=f(0)+f'(0)+\int_{0}^{1}\int_{0}^{1}tf''(st)dsdt.
\]
Hence, by combining (\ref{eq:2-1111}) and (\ref{eq:2-114}), we conclude with the claimed formula. 
\end{proof}
\subsection{Proof of Lemma \ref{lemma_2_3}}
\begin{proof}
Due to the symmetry of the matrix $D_{x}^{2}\dfrac{\partial u}{\partial m}(\mathcal{L}_{X})(x)$ for each $x \in \mathbb{R}^n$, we have
\[
D_{x}^{2}\dfrac{\partial u}{\partial m}(\mathcal{L}_{X})(X)Z \cdot Y=D_{x}^{2}\dfrac{\partial u}{\partial m}(\mathcal{L}_{X})(X)Y \cdot Z.
\]
On the other hand, we also have: 
\small{\begin{eqnarray*}
\mathbb{E}\left(D_{\xi}D_{\eta}\dfrac{\partial^{2}u}{\partial m^{2}}(\mathcal{L}_{X})(X,\tilde{X})\tilde{Z} \cdot Y \right)&=&\sum_{i,j}\mathbb{E} \left( D_{\xi_{i}}D_{\eta_{j}} \dfrac{\partial^{2}u}{\partial m^{2}}(\mathcal{L}_{X})(X,\tilde{X})\tilde{Z}_{j}Y_{i}\right) \\
&=& \sum_{i,j}\mathbb{E} \left(D_{\xi_{j}}D_{\eta_{i}}\dfrac{\partial^{2}u}{\partial m^{2}}(\mathcal{L}_{X})(X,\tilde{X})\tilde{Z}_{i}Y_{j}\right) \\
&=&\sum_{i,j}\mathbb{E} \left( D_{\eta_{j}}D_{\xi_{i}}\dfrac{\partial^{2}u}{\partial m^{2}}(\mathcal{L}_{X})(\tilde{X},X)\tilde{Z}_{i}Y_{j} \right) \\
&=& \sum_{i,j}\mathbb{E} \left( D_{\eta_{j}}D_{\xi_{i}}\dfrac{\partial^{2}u}{\partial m^{2}}(\mathcal{L}_{X})(X,\tilde{X})Z_{i}\tilde{Y}_{j} \right) =\mathbb{E} \left(D_{\xi}D_{\eta}\dfrac{\partial^{2}u}{\partial m^{2}}(\mathcal{L}_{X})(X,\tilde{X})\tilde{Y} \cdot Z\right) ,
 \end{eqnarray*}}
where the third equality follows by recalling another symmetry property: $\dfrac{\partial^{2}u}{\partial m^{2}}(\mathcal{L}_{X})(X,\tilde{X})=\dfrac{\partial^{2}u}{\partial m^{2}}(\mathcal{L}_{X})(\tilde{X},X)$; the fourth equality follows by noting that $(\tilde{X},\tilde{Y},\tilde{Z})$ is an independent copy of $(X,Y,Z)$. Hence, the symmetry of the bilinear functional is concluded. 
\end{proof}
\section{Proofs of the statements in Section \ref{ABSTRACT_CONTROL_PROBLEM}}
\subsection{Proof of Proposition \ref{prop:2-1}}
\begin{proof} Let $\tilde{v}(\cdot) \in L_{\mathcal{W}_{X,t}}^{2}(0,T;\mathcal{H})$, and we consider the perturbed objective functional $J_{X,t}(v+\theta\tilde{v})$. One has the
formula
\begin{align*}
J_{X,t}(v+\theta\tilde{v}) = & J_{X,t}(v)+ \dfrac{\lambda}{2}\, \theta^{2}\int_{t}^{T}||\tilde{v}(s)||^{2}ds \\
& + \theta\left\{ \int_{t}^{T}((\lambda v(s),\tilde{v}(s)))ds+\int_{t}^{T}((D_{X}F(X(s)),\int_{t}^{s}\tilde{v}(\tau)d\tau))ds+((D_{X}F_{T}(X(T)),\int_{t}^{T}\tilde{v}(\tau)d\tau))\right\} \\
& +\theta\int_{t}^{T}\int_{0}^{1}((D_{X}F\left(X(s)+\mu\theta\int_{t}^{s}\tilde{v}(\tau)d\tau\right)-D_{X}F(X(s)),\int_{t}^{s}\tilde{v}(\tau)d\tau)) d\mu ds \\
& + \theta\int_{0}^{1}((D_{X}F_{T}\left(X(T)+\mu\theta\int_{t}^{T}\tilde{v}(\tau)d\tau\right)-D_{X}F(X(T)),\int_{t}^{T}\tilde{v}(\tau)d\tau))d\mu .
\end{align*}
According to the assumptions (\ref{eq:2-1}), we obtain, as $\theta \rightarrow 0$,
\begin{align*}
\left|\int_{t}^{T}\int_{0}^{1}((D_{X}F\left(X(s)+\mu\theta\int_{t}^{s}\tilde{v}(\tau)d\tau\right)-D_{X}F(X(s)),\int_{t}^{s}\tilde{v}(\tau)d\tau))dsd\mu\right| & \leq \dfrac{c}{2} \, \theta\int_{t}^{T} \left| \int_{t}^{s}\tilde{v}(\tau)d\tau \right|^{2}ds \rightarrow 0 ; 
\end{align*} 
similarly, we can also see the convergence of $\int_{0}^{1}((D_{X}F_{T}\left(X(T)+\mu\theta\int_{t}^{T}\tilde{v}(\tau)d\tau\right)-D_{X}F(X(T)),\int_{t}^{T}\tilde{v}(\tau)d\tau))d\mu$ to $0$ as $\theta$ tends to $0$. Therefore, we get by using integration by parts, 
\begin{align*}
& \quad \lim_{\theta \rightarrow 0} \dfrac{J_{X,t}(v+\theta\tilde{v})-J_{X,t}(v)}{\theta} \\
& = \int_{t}^{T}((\lambda v(s),\tilde{v}(s)))ds
+\int_{t}^{T}((D_{X}F(X(s)),\int_{t}^{s}\tilde{v}(\tau)d\tau))ds+((D_{X}F_{T}(X(T)),\int_{t}^{T}\tilde{v}(\tau)d\tau)) \\
& =\int_{t}^{T}((\lambda v(s)+ \mathbb{E}\left[ D_{X}F_{T}(X(T))+\int_{s}^{T}D_{X}F(X(\tau))d\tau \Bigg| \mathcal{F}^{s} \right],\tilde{v}(s)))ds ,
\end{align*}
where the last line follows by a simple application of the tower property, and therefore the result (\ref{eq:2-8}) follows.
\end{proof}

\subsection{Proof of Proposition \ref{prop:2-2}}
\begin{proof}
Using $X_{1}(s)$ and $X_{2}(s)$ to represent the states corresponding to the controls $v_{1}$ and $v_{2}$ respectively, and invoking the formula (\ref{eq:2-8}) and the fact that both $v_{1}(s)$ and $v_{2}(s)$ are adapted to $\mathcal{W}_{X,t}^{s}$, 
\begin{align*}
& \int_{t}^{T}((D_{v}J_{X,t}(v_{1})(s)-D_{v}J_{X,t}(v_{2})(s),v_{1}(s)-v_{2}(s)\,))ds \\
= & \lambda\int_{t}^{T}||v_{1}(s)-v_{2}(s)||^{2}ds + \int_{t}^{T}((D_{X}F_{T}(X_{1}(T))-D_{X}F_{T}(X_{2}(T)) \\
& + \int_{s}^{T}(D_{X}F(X_{1}(\tau))-D_{X}F(X_{2}(\tau))) d\tau,\dfrac{d}{ds}(X_{1}(s)-X_{2}(s))\,))ds \\
= & \lambda\int_{t}^{T}||v_{1}(s)-v_{2}(s)||^{2}ds +((D_{X}F_{T}(X_{1}(T))-D_{X}F_{T}(X_{2}(T)),X_{1}(T)-X_{2}(T)\,)) \\
& +\int_{t}^{T}((D_{X}F(X_{1}(s))-D_{X}F(X_{2}(s)),X_{1}(s)-X_{2}(s)\,))ds \\
\geq & \lambda\int_{t}^{T}||v_{1}(s)-v_{2}(s)||^{2}ds-c'_{T}||X_{1}(T)-X_{2}(T)||^{2}-c'\int_{t}^{T}||X_{1}(s)-X_{2}(s)||^{2}ds  ,
\end{align*}
where the last inequality follows due to the assumption (\ref{eq:2-3}). Since $X_{1}(s)-X_{2}(s)=\int_{t}^{s}(v_{1}(\tau)-v_{2}(\tau))d\tau$, we get immediately the estimates: 
\begin{align*}
||X_{1}(T)-X_{2}(T)||^{2} \leq T \int_{t}^{T}||v_{1}(s)-v_{2}(s)||^{2}ds \text{ and } \int_{t}^{T}||X_{1}(s)-X_{2}(s)||^{2}ds \leq  \dfrac{T^{2}}{2}\int_{t}^{T}||v_{1}(s)-v_{2}(s)||^{2}ds .
\end{align*}
Therefore, 
\[
\int_{t}^{T}((D_{v}J_{X,t}(v_{1})(s)-D_{v}J_{X,t}(v_{2})(s),v_{1}(s)-v_{2}(s)\,))ds\geq \left(\lambda-c'T-c'_{T}\dfrac{T^{2}}{2}\right) \int_{t}^{T}||v_{1}(s)-v_{2}(s)||^{2}ds ,
\]
and the claim (\ref{eq:2-10}) is obtained. 
\end{proof}

\section{Proofs of the statements in Section \ref{sec:STUDY-OF-THE}}
\subsection{Proof of Proposition \ref{prop:3-4}}
\begin{proof} For simplicity, we omit the subscripts of $X$ and $t$ in $Y$ and $Z$. Let $\Upsilon(s)=D_{X}F_{T}(Y(T))+\int_{s}^{T}D_{X}F(Y(\tau))d\tau$, so $Z(s)=\mathbb{E}[\Upsilon(s)|\mathcal{W}_{X,t}^{s}]$. By the tower property, we have $((\Upsilon(s),Y(s)\,))=((Z(s),Y(s)\,))$. Then,
\[
\dfrac{d}{ds}((Z(s),Y(s)\,))=\dfrac{d}{ds}((\Upsilon(s),Y(s)\,))=-\dfrac{1}{\lambda}||Z(s)||^{2}-((D_{X}F(Y(s)),Y(s)\,)) ,
\]
and hence, by integrating the last equation from $t$ to $T$, we obtain: 
\begin{equation}
((X,Z(t)))=\dfrac{1}{\lambda}\int_{t}^{T}||Z(s)||^{2}ds+((D_{X}F_{T}(Y(T)),Y(T)))+\int_{t}^{T}((D_{X}F(Y(s)),Y(s)))ds . \label{eq:3-4}
\end{equation}
On the other hand, using the tower property again, by definition: 
\[
((X,Z(t)))=((X,D_{X}F_{T}(Y(T))+\int_{t}^{T}D_{X}F(Y(s))ds)) ,
\]
and then combining this with (\ref{eq:3-4}) and then telescoping, we have
\begin{align*}
& \quad \dfrac{1}{\lambda}\int_{t}^{T}||Z(s)||^{2}ds+((D_{X}F_{T}(Y(T))-D_{X}F_{T}(0),Y(T)))+\int_{t}^{T}((D_{X}F(Y(s))-D_{X}F(0),Y(s)))ds \\
& = ((X,D_{X}F_{T}(Y(T))-D_{X}F_{T}(0)+\int_{t}^{T}(D_{X}F(Y(s))-D_{X}F(0))ds)) \\
& \quad + ((X-Y(T),D_{X}F_{T}(0)))+\int_{t}^{T}((X-Y(s),D_{X}F(0)))ds .
\end{align*}
Thanks to the measurability assumption (\ref{eq:2-4}), both $D_{X}F_{T}(0)$ and $D_{X}F(0)$ are deterministic, hence we have by using the expression of $Y(T)$ in (\ref{eq:3-1}), 
\begin{align*} 
& \quad ((X-Y(T),D_{X}F_{T}(0)))+\int_{t}^{T}((X-Y(s),D_{X}F(0)))ds \nonumber \\
& = \dfrac{1}{\lambda}((D_{X}F_{T}(0),\int_{t}^{T}Z(\tau)d\tau))+\dfrac{1}{\lambda}\int_{t}^{T}((D_{X}F(0),\int_{t}^{s}Z(\tau)d\tau))ds , 
\end{align*}
and a simple application of Cauchy-Schwartz's inequality gives:
\begin{align} \label{prop 6 inequality 1}
& \quad \left|((X-Y(T),D_{X}F_{T}(0)))+\int_{t}^{T}((X-Y(s),D_{X}F(0)))ds \right| \nonumber \\
& \leq \dfrac{1}{\lambda}\sqrt{T} \left( ||D_{X}F_{T}(0)||+\dfrac{2}{3}T||D_{X}F(0)|| \right) \sqrt{\int_{t}^{T}||Z(s)||^{2}ds} .
\end{align}
In addition, using Lipschitz property (\ref{eq:2-1}), we have
\begin{align} \label{prop 6 inequality 2}
& \quad \left| ((X,D_{X}F_{T}(Y(T))-D_{X}F_{T}(0)+\int_{t}^{T}(D_{X}F(Y(s))-D_{X}F(0))ds)) \right| \nonumber \\
& \leq \left( c_{T}||Y(T)||+c\int_{t}^{T}||Y(s)||ds \right)||X|| \nonumber \\
& \leq (c_{T}+cT)||X||^{2}+\sqrt{T} \left( c_{T}+\dfrac{2}{3}cT \right)||X|| \left( ||\sigma||+ \frac{1}{\lambda} \sqrt{\int_{t}^{T}||Z(s)||^{2}ds} \right) ,
\end{align}
where the last inequality follows due to the fact that $||Y(s)|| \leq ||X|| + \frac{1}{\lambda} \sqrt{s} \sqrt{\int_t^s ||Z(\tau)||^2 d\tau} + ||\sigma|| \sqrt{s}$. On the other hand, using the quasi-convexity assumption (\ref{eq:2-3}), we also have:
\begin{align*}
& \quad \dfrac{1}{\lambda}\int_{t}^{T}||Z(s)||^{2}ds+((D_{X}F_{T}(Y(T))-D_{X}F_{T}(0),Y(T)))+\int_{t}^{T}((D_{X}F(Y(s))-D_{X}F(0),Y(s)))ds  \\
& \geq \dfrac{1}{\lambda}\int_{t}^{T}||Z(s)||^{2}ds-c'_{T}||Y(T)||^{2}-c'\int_{t}^{T}||Y(s)||^{2}ds .
\end{align*}
Note that, we can also have, for any $\epsilon>0$, 
\[
||Y(s)||^{2}\leq(||X||^{2}+||\sigma||^{2}(s-t)) \left( 1+\dfrac{1}{\epsilon} \right)+ (1+2\epsilon)\dfrac{1}{\lambda^{2}}(s-t)\int_{t}^{T}||Z(\tau)||^{2}d\tau ,
\]
with which we deduce that
\begin{align} \label{prop 6 inequality 3}
& \quad \dfrac{1}{\lambda}\int_{t}^{T}||Z(s)||^{2}ds-c'_{T}||Y(T)||^{2}-c'\int_{t}^{T}||Y(s)||^{2}ds \\
& \geq \dfrac{1}{\lambda} \left( 1-\dfrac{1}{\lambda}(1+2\epsilon)T \left( c'_{T}+c'\dfrac{T}{2} \right) \right) \int_{t}^{T}||Z(\tau)||^{2}d\tau-(c'_{T}+c'T)\left( 1 + \frac{1}{\epsilon} \right)||X||^{2}-||\sigma||^{2} \left( 1+\dfrac{1}{\epsilon} \right) T \left( c'_{T}+\dfrac{c'}{2}T \right) . \nonumber
\end{align}
Thanks to the assumption (\ref{eq:2-9}), we can find sufficiently small $\epsilon$ so that $\lambda-(1+2\epsilon)T \left( c'_{T}+c'\dfrac{T}{2} \right)>0$. From the preceding inequalities (\ref{prop 6 inequality 1}),(\ref{prop 6 inequality 2}) and (\ref{prop 6 inequality 3}), and then combining them, we obtain
\begin{align*}
& \quad \dfrac{1}{\lambda} \left( 1-\dfrac{1}{\lambda}(1+2\epsilon)T \left( c'_{T}+c'\dfrac{T}{2} \right) \right) \int_{t}^{T}||Z(\tau)||^{2}d\tau-(c'_{T}+c'T)\left( 1 + \frac{1}{\epsilon} \right)||X||^{2}-||\sigma||^{2} \left( 1+\dfrac{1}{\epsilon} \right) T \left( c'_{T}+\dfrac{c'}{2}T \right)  \\
& \leq (c_{T}+cT)||X||^{2}+\sqrt{T} \left( c_{T}+\dfrac{2}{3}cT \right)||X|| \left( ||\sigma||+ \frac{1}{\lambda} \sqrt{\int_{t}^{T}||Z(s)||^{2}ds} \right) \\
& \quad + \dfrac{1}{\lambda}\sqrt{T} \left( ||D_{X}F_{T}(0)||+\dfrac{2}{3}T||D_{X}F(0)|| \right) \sqrt{\int_{t}^{T}||Z(s)||^{2}ds} . 
\end{align*}
By solving this quadratic inequality we deduce that
\[
\int_{t}^{T}||Z(s)||^{2}ds\leq C(1+||X||^{2}),
\]
where $C$ is a generic constant depending on different constants of the model, but not on $X$; based on this, the first estimate (\ref{eq:3-3}) follows from the definition of $Y(s)$ in (\ref{eq:3-1}). While the second estimate follows from the definition of $Z(s)$ in (\ref{eq:3-1}) and the assumption (\ref{eq:2-2}). The third estimate for the value function in (\ref{eq:3-3}) follows immediately from the formula (\ref{eq:3-2}) and again the assumption (\ref{eq:2-2}). This concludes the proof.
\end{proof}

\subsection{Proof of Theorem \ref{Theo3-1}}
\begin{proof}
Let $X_{1}, X_{2} \in \mathcal{H}$ be $\mathcal{F}^{t}$-measurable. Consider the functionals $J_{X_{1},t}(v)$ and $J_{X_{2},t}(v)$, and denote by $Y_{1}(s)$ and $Y_{2}(s)$ the optimal states corresponding to the respective optimal controls $u_{1}(s)$ and $u_{2}(s)$ for these two functionals. Then 
\[
u_{1}(s)=-\dfrac{1}{\lambda}Z_{1}(s) \text{ and } u_{2}(s)=-\dfrac{1}{\lambda}Z_{2}(s) ,
\]
where $Z_i(s) = \mathbb{E} \left[  D_XF_T(Y_i(T)) + \int_s^T D_X F(Y_i(\tau)) d\tau \Bigg| \mathcal{W}_{X,t}^{s} \right]$ for $i=1,2$. It is clear that 
\[
V(X_{1},t)-V(X_{2},t)\leq J_{X_{1}t}(u_{2})-J_{X_{2}t}(u_{2}) .
\]
Noting that the trajectory starting from $(X_{1},t)$ subject to $u_{2}$ is simply $Y_{2}(s)+X_{1}-X_{2}$, we get immediately that 
\[
V(X_{1},t)-V(X_{2},t)\leq F_{T}(Y_{2}(T)+X_{1}-X_{2})-F_{T}(Y_{2}(T))+\int_{t}^{T}(F(Y_{2}(s)+X_{1}-X_{2})-F(Y_{2}(s)))ds .
\]
According to the assumptions (\ref{eq:2-1}), simple mean-value argument yields
\begin{align*}
& |F(Y_{2}(s)+X_{1}-X_{2})-F(Y_{2}(s))-((D_{X}F(Y_{2}(s)),X_{1}-X_{2}))|\leq c||X_{1}-X_{2}||^{2} , \\
& |F_{T}(Y_{2}(T)+X_{1}-X_{2})-F_{T}(Y_{2}(T))-((D_{X}F_{T}(Y_{2}(T)),X_{1}-X_{2}))|\leq c_{T}||X_{1}-X_{2}||^{2} ,
\end{align*}
from which it follows that
\begin{equation}
V(X_{1},t)-V(X_{2},t)\leq((Z_{2}(t),X_{1}-X_{2}))+(c_{T}+cT)||X_{1}-X_{2}||^{2} . \label{eq:3-5}
\end{equation}
By interchanging the roles of $X_{1}$ and $X_{2}$, we also have:
\[
V(X_{2},t)-V(X_{1},t)\leq((Z_{1}(t),X_{2}-X_{1}))+(c_{T}+cT)||X_{1}-X_{2}||^{2} ,
\]
and hence, 
\begin{equation}
V(X_{1},t)-V(X_{2},t)\geq((Z_{2}(t),X_{1}-X_{2}))-(c_{T}+cT)||X_{1}-X_{2}||^{2}+((Z_{1}(t)-Z_{2}(t),X_{1}-X_{2})) .\label{eq:3-6}
\end{equation}
Define $\Upsilon_{1}(s):=D_{X}F_{T}(Y_{1}(T))+\int_{s}^{T}D_{X}F(Y_{1}(\tau))d\tau$ and $\Upsilon_{2}(s):=D_{X}F_{T}(Y_{2}(T))+\int_{s}^{T}D_{X}F(Y_{2}(\tau))d\tau$, and then use them as that in the proof of Proposition \ref{prop:3-4}, 
\begin{align*}
((Z_{1}(t)-Z_{2}(t),X_{1}-X_{2}))= & \dfrac{1}{\lambda} \int_{t}^{T}||Z_{1}(s)-Z_{2}(s)||^{2}ds +\int_{t}^{T}((D_{X}F(Y_{1}(s))-D_{X}F(Y_{2}(s)),Y_{1}(s)-Y_{2}(s)))ds \\
& +((D_{X}F_{T}(Y_{1}(T))-D_{X}F_{T}(Y_{2}(T)),Y_{1}(T)-Y_{2}(T))) .
\end{align*}
From assumption (\ref{eq:2-3}), it follows that 
\begin{align} \label{eq:3-61}
((Z_{1}(t)-Z_{2}(t),X_{1}-X_{2})) \geq & \dfrac{1}{\lambda}\int_{t}^{T}||Z_{1}(s)-Z_{2}(s)||^{2}ds-c'_{T}||Y_{1}(T)-Y_{2}(T)||^{2} \nonumber \\
& -c'\int_{t}^{T}||Y_{1}(s)-Y_{2}(s)||^{2}ds .
\end{align}
Proceeding as in the proof of Proposition \ref{prop:3-4}, we obtain the inequality, for any $\epsilon > 0$,
\begin{align*}
((Z_{1}(t)-Z_{2}(t),X_{1}-X_{2})) \geq & \dfrac{1}{\lambda} \left( 1-\dfrac{T}{\lambda}(1+\epsilon) \left( c'_{T}+c'\dfrac{T}{2} \right) \right)\int_{t}^{T}||Z_{1}(s)-Z_{2}(s)||^{2}ds  \\
& -(c'_{T}+c'T) \left( 1+\dfrac{1}{\epsilon} \right) ||X_{1}-X_{2}||^{2} .
\end{align*}
Thanks to assumption (\ref{eq:2-9}), we can find a sufficiently small $\epsilon$ so that $1-\dfrac{T}{\lambda}(1+\epsilon)\left( c'_{T}+c'\dfrac{T}{2} \right)>0$, from which it follows that
\[
((Z_{1}(t)-Z_{2}(t),X_{1}-X_{2}))\geq -(c'_{T}+c'T) \left( 1+\dfrac{1}{\epsilon} \right) ||X_{1}-X_{2}||^{2} ;
\]
Combining with (\ref{eq:3-6}), we obtain 
\begin{equation}
V(X_{1},t)-V(X_{2},t)\geq((Z_{2}(t),X_{1}-X_{2}))- \left[ (c_{T}+cT)+(c'_{T}+c'T) \left( 1+\dfrac{1}{\epsilon} \right) \right] ||X_{1}-X_{2}||^{2} ; \label{eq:3-7}
\end{equation}
together with the inequality (\ref{eq:3-5}), we can obtain
\[
|V(X_{1},t)-V(X_{2},t)-((Z_{2}(t),X_{1}-X_{2}))|\leq C||X_{1}-X_{2}||^{2} ,
\]
for some constant $C>0$. This implies that $V(X,t)$ has a G\^ateaux derivative (and even a Fr\'echet derivative) at any argument $X \in  \mathcal{H}$ which is $\mathcal{F}^{t}$-measurable and $D_{X}V(X,t)=Z_{X,t}(t)$, and hence it is $\sigma(X)$-measurable in accordance with Proposition \ref{prop:3-31}. Moreover the gradient satisfies the first estimate from (\ref{eq:3-31}) in accordance with (\ref{eq:3-3}) in Proposition \ref{prop:3-4}. We next establish the second estimate from (\ref{eq:3-31}).
The tower property gives
\begin{align*}
((Z_{1}(t)-Z_{2}(t),X_{1}-X_{2}))= & ((X_{1}-X_{2},D_{X}F_{T}(Y_{1}(T))-D_{X}F_{T}(Y_{2}(T)) \\
& +\int_{t}^{T}(D_{X}F(Y_{1}(s))-D_{X}F(Y_{2}(s)))ds\,)) ;
\end{align*}
and hence, 
\begin{align*}
|((Z_{1}(t)-Z_{2}(t),X_{1}-X_{2}))| \leq & ||X_{1}-X_{2}||\,\left(c_{T}||Y_{1}(T)-Y_{2}(T)||+c\int_{t}^{T}||Y_{1}(s)-Y_{2}(s)||ds \right) \\
\leq & ||X_{1}-X_{2}||\,\left[(c_{T}+cT)||X_{1}-X_{2}||+\dfrac{c_{T}}{\lambda}\int_{t}^{T}||Z_{1}(s)-Z_{2}(s)||ds \right. \\
& \quad + \left. \dfrac{c}{\lambda}\int_{t}^{T}\left( \int_{t}^{s}||Z_{1}(\tau)-Z_{2}(\tau)||d\tau \right) ds \right] .
\end{align*}
After similar calculations as before, and then combining with (\ref{eq:3-61}), we obtain, for any $\epsilon > 0$, 
\begin{align*}
\dfrac{1}{\lambda}\int_{t}^{T}||Z_{1}(s)-Z_{2}(s)||^{2}ds \leq & (c_{T}+cT)\left( 1+\dfrac{1}{2\lambda\epsilon} \right)||X_{1}-X_{2}||^{2} + \dfrac{T\epsilon}{2\lambda} \left( c_{T}+\dfrac{cT}{2} \right) \int_{t}^{T}||Z_{1}(s)-Z_{2}(s)||^{2}ds \\
& + c'_{T}||Y_{1}(T)-Y_{2}(T)||^{2} + c'\int_{t}^{T}||Y_{1}(s)-Y_{2}(s)||^{2}ds .
\end{align*}
Finally, we obtain the estimate 
\begin{align} \label{eq:3-71}
& \quad \dfrac{1}{\lambda} \left[ 1-\dfrac{T\epsilon}{2} \left( c_{T}+\dfrac{cT}{2} \right)-\dfrac{T}{\lambda}(1+\epsilon)\left( c'_{T}+\dfrac{c'T}{2} \right) \right] \int_{t}^{T}||Z_{1}(s)-Z_{2}(s)||^{2}ds \\
& \leq \left[ (c_{T}+cT)\left( 1+\dfrac{1}{2\lambda} \right)+(c'_{T}+c'T)\left( 1+\dfrac{1}{\epsilon} \right) \right] ||X_{1}-X_{2}||^{2} . \nonumber 
\end{align}
As previously, the condition (\ref{eq:2-9}) ensures that there is a sufficiently small $\epsilon$, such that we can have
\[
\int_{t}^{T}||Z_{1}(s)-Z_{2}(s)||^{2}ds\leq C||X_{1}-X_{2}||^{2},
\]
for some constant $C>0$. By using the expressions in (\ref{eq:3-1}), we deduce that
\[
\sup_{t}||Y_{1}(s)-Y_{2}(s)||,\;\sup_{t\leq s\leq T}||Z_{1}(s)-Z_{2}(s)||\leq C||X_{1}-X_{2}|| .
\]
Recalling that $Z_{1}(t)=D_{X}V(X_{1},t)$ and $Z_{2}(t)=D_{X}V(X_{2},t)$, the second estimate (\ref{eq:3-31}) follows accordingly. 
\end{proof}

\subsection{Proof of Proposition \ref{prop:3-6}}
\begin{proof}
According to the optimality principle (\ref{eq:3-8}), we have
\[
V(X,t)-V(X,t+h)=\dfrac{1}{2\lambda}\int_{t}^{t+h}||Z(s)||^{2}ds+\int_{t}^{t+h}F(Y(s))ds+V(Y(t+h),t+h)-V(X,t+h) ;
\]
From the differentiability of $V$, we get
\begin{align*}
V(Y(t+h),t+h)-V(X,t+h) & =\int_{0}^{1}((D_{X}V(X+\theta(Y(t+h)-X),t+h),Y(t+h)-X))d\theta \\
& =-\dfrac{1}{\lambda}((D_{X}V(X,t+h),\int_{t}^{t+h}Z(s)ds)) \\
& \quad +\int_{0}^{1}((D_{X}V(X+\theta(Y(t+h)-X),t+h)-D_{X}V(X,t+h),Y(t+h)-X)) ,
\end{align*}
where for the second equality, we have used the fact that $D_{X}V(X,t+h)$ is independent of $\mathcal{W}_t^{t+h}$, since $D_{X}V(X,t+h)$ is $\sigma(X)$-measurable in accordance with Proposition \ref{prop:3-31}. Using also
the Lipschitz property (\ref{eq:3-31}) in Theorem \ref{Theo3-1} of the gradient of the value function, together with the assumption (\ref{eq:2-2}) and (\ref{eq:3-4}) in Proposition \ref{prop:3-4}, we obtain easily that 
\[
||V(X,t)-V(X,t+h)||\leq Ch(1+||X||^{2}) , \text{ for any } 0 \leq h \leq T, 
\]
which is the result (\ref{eq:3-9}). To establish the H\"older continuity in time of $D_{X}V(X,t)$, we first prove a useful esimate. Fix two times $0 \leq t_{1} < t_{2} \leq T$. Take $X$ to be $\mathcal{F}^{t_{1}}$-measurable. For notational simplicity, we take $Y_{X,t_{1}}(s)$ and $Y_{X,t_{2}}(s)$ being denoted respectively by $Y_{1}(s)$ and $Y_{2}(s)$; and similarly $Z_{X,t_{1}}(s)$ by $Z_{1}(s)$ and $Z_{X,t_{2}}(s)$ by $Z_{2}(s)$. We claim that: 
\begin{align}
& \sup_{t_{2}\leq s\leq T}||Y_{X,t_{1}}(s)-Y_{X,t_{2}}(s)||\leq C(1+||X||)(t_{2}-t_{1})^{\frac{1}{2}} , \label{eq:3-10} \\
& \sup_{t_{2}\leq s\leq T}||Z_{X,t_{1}}(s)-Z_{X,t_{2}}(s)||\leq C(1+||X||)(t_{2}-t_{1})^{\frac{1}{2}} . \label{eq:3-101}
\end{align}
These estimates are derived by the reasoning as in the proof of Theorem \ref{Theo3-1}, and we just sketch the key idea here; so we first obtain 
\begin{align*}
((Z_{2}(t_{2})-Z_{1}(t_{2}),X-Y_{1}(t_{2})))= & \dfrac{1}{\lambda}\int_{t_{2}}^{T}||Z_{2}(s)-Z_{1}(s)||^{2}ds +((Y_{2}(T)-Y_{1}(T),D_{X}F_{T}(Y_{2}(T))-D_{X}F_{T}(Y_{1}(T))\,)) \\
& +\int_{t_{2}}^{T}((Y_{2}(s)-Y_{1}(s),D_{X}F(Y_{2}(s))-D_{X}F(Y_{1}(s))\,))ds \\
\geq & \dfrac{1}{\lambda}\int_{t_{2}}^{T}||Z_{2}(s)-Z_{1}(s)||^{2}ds-c'_{T}||Y_{2}(T)-Y_{1}(T)||^{2}-c'\int_{t_{2}}^{T}||Y_{2}(s)-Y_{1}(s)||^{2}ds ,  
\end{align*}
and similar to (\ref{eq:3-71}), it follows that 
\begin{align*}
& \dfrac{1}{\lambda}\left[1-\dfrac{T\epsilon}{2}\left(c_{T}+\dfrac{cT}{2}\right)-\dfrac{T}{\lambda}(1+\epsilon)\left(c'_{T}+\dfrac{c'T}{2}\right)\right] \int_{t_{2}}^{T}||Z_{1}(s)-Z_{2}(s)||^{2}ds \\
\leq & 
\left[(c_{T}+cT)\left(1+\dfrac{1}{2\lambda}\right)+(c'_{T}+c'T)\left(1+\dfrac{1}{\epsilon}\right)\right]||X-Y_{1}(t_{2})||^{2} ,
\end{align*}
from which we derive $\int_{t_{2}}^{T}||Z_{1}(s)-Z_{2}(s)||^{2}ds\leq C||X-Y_{1}(t_{2})||^{2}$ for some $C>0$, and by using the expression (\ref{eq:3-1}), we see that 
\[
\sup_{t_{2}\leq s\leq T}||Y_{1}(s)-Y_{2}(s)||^{2}\leq C||X-Y_{1}(t_{2})||^{2},\:\sup_{t_{2}\leq s\leq T}||Z_{1}(s)-Z_{2}(s)||^{2}\leq C||X-Y_{1}(t_{2})||^{2} ;
\]
also noting that $X-Y_{1}(t_{2})=\dfrac{1}{\lambda}\int_{t_{1}}^{t_{2}}Z_{1}(s)ds-\sigma(w(t_{2})-w(t_{1}))$, we further obtain  the estimates (\ref{eq:3-10}),(\ref{eq:3-101}). Next, we establish the H\"older continuity from the left in time of $D_X V(X,t)$. Let $t_{n}\uparrow t$, and set $Y^{n}(s)=Y_{X,t_{n}}(s)$, $Z^{n}(s)=Z_{X,t_{n}}(s)$ and $Y(s)=Y_{X,t}(s),$ $Z(s)=Z_{X,t}(s)$.
As a consequence of (\ref{eq:3-10}), it follows that 
\[
\sup_{t\leq s\leq T}||Y(s)-Y^{n}(s)||\leq C(1+||X||)(t-t_{n})^{\frac{1}{2}},\:\sup_{t\leq s\leq T}||Z(s)-Z^{n}(s)||\leq C(1+||X||)(t-t_{n})^{\frac{1}{2}} .
\]
We then telescope the term:
\begin{align*}
Z^{n}(t_{n})= & \mathbb{E}\left[ D_{X}F_{T}(Y(T))+\int_{t}^{T}D_{X}F(Y(s))ds \Bigg| \,\mathcal{W}_{X,t}^{t_{n}}\right] +\mathbb{E}\left[ \int_{t_{n}}^{t}D_{X}F(Y^{n}(s))ds \Bigg| \,\mathcal{W}_{X,t}^{t_{n}}\right] \\
& +\mathbb{E}\left[D_{X}F_{T}(Y^{n}(T))-D_{X}F_{T}(Y(T))+\int_{t}^{T}(D_{X}F(Y^{n}(s))-D_{X}F(Y(s)))ds \Bigg| \,\mathcal{W}_{X,t}^{t_{n}}\right] , 
\end{align*}
and from the previous estimates, due to the Lipschitz property of $D_X V(X,t)$, it remains to observe that $\mathbb{E}\left[D_{X}F_{T}(Y(T))+\int_{t}^{T}D_{X}F(Y(s))ds \Bigg| \,\mathcal{W}_{X,t}^{t_{n}}\right]\rightarrow Z(t)$ as $t_{n}\uparrow t$; indeed $\mathbb{E}\left[ D_{X}F_{T}(Y(T))+\int_{t}^{T}D_{X}F(Y(s))ds \Bigg| \,\mathcal{W}_{X,t}^{u} \right]$ is  a continuous martingale in $u$. Finally, for the continuity from the right $t_{n}\downarrow t$, from (\ref{eq:3-101}), we first have $||Z^{n}(t_{n})-Z(t_{n})||\leq C(1+||X||)(t_{n}-t)^{\frac{1}{2}}$. Next 
\[
Z(t_{n})+\int_{t}^{t_{n}}D_{X}F(Y(s))ds=\mathbb{E}\left[\int_{t}^{T}D_{X}F(Y(\tau))d\tau+D_{X}F_{T}(Y(T)) \Bigg| \mathcal{W}_{X,t}^{t_{n}} \right] ,
\]
and the property follows again from the continuity property of the martingale with respect to the filtration $\mathcal{W}_{X,t}^{u}$ as above. 
\end{proof}

\end{document}